\documentclass[letterpaper,11pt]{article}

\usepackage{latexsym}
\usepackage[ruled,vlined]{algorithm2e}
\usepackage{verbatim}

\usepackage{amsmath}
\usepackage{amssymb}
\usepackage{amsfonts}

\newtheorem{theorem}{Theorem}[section]
\newtheorem{lemma}[theorem]{Lemma}

\newtheorem{corollary}[theorem]{Corollary}
\newtheorem{definition}[theorem]{Definition}

\newenvironment{proof}{{\bf Proof:\ }}{\hfill$\Box$\medskip}

\newcommand{\Exp}{\mathbb{E}}
\renewcommand{\Pr}{\mathbb{P}}
\newcommand{\Var}{{\rm Var}}

\newcommand{\ee}{{\rm e}}

\newcommand{\weight}{c}
\newcommand{\SP}{\rightsquigarrow}
\newcommand{\LSP}{\mbox{$\cal LSP$}}
\newcommand{\SPP}{\mbox{$\cal SP$}}

\newcommand{\ignore}[1]{}
\newcommand{\remove}[1]{}
\newcommand{\reals}{\hbox{$\rlap{\rm I} \> \kern-.2mm{\rm R}$}}

\newcommand{\etal}{{\em et al.\ }}

\newcommand{\apsp}{\mbox{{\tt apsp}}}
\newcommand{\heap}{\mbox{{\tt heap}}}
\newcommand{\extractmin}{\mbox{{\tt extract}-{\tt min}}}
\newcommand{\decreasekey}{\mbox{{\tt decrease}-{\tt key}}}
\newcommand{\findmin}{\mbox{{\tt find}-{\tt min}}}
\newcommand{\INSERT}{\mbox{{\tt insert}}}
\newcommand{\DELETE}{\mbox{{\tt delete}}}
\newcommand{\heapinsert}{\mbox{{\tt heap}-{\tt insert}}}
\newcommand{\heapdelete}{\mbox{{\tt heap}-{\tt delete}}}

\newcommand{\PATH}{\mbox{{\tt path}}}
\newcommand{\newpath}{\mbox{{\tt new}-{\tt path}}}

\newcommand{\EXAMINE}{\mbox{{\tt examine}}}
\newcommand{\examine}{\mbox{{\tt examine}}}
\newcommand{\init}{\mbox{{\tt init}}}
\newcommand{\dapspinit}{\mbox{{\tt dapsp}}-\mbox{{\tt init}}}

\newcommand{\FALSE}{{\tt false}}
\newcommand{\TRUE}{{\tt true}}

\newcommand{\NULL}{null}

\newcommand{\buildpaths}{\mbox{{\tt build}-{\tt paths}}}

\newcommand{\removepath}{\mbox{{\tt remove}-{\tt path}}}
\newcommand{\replacepath}{\mbox{{\tt replace}-{\tt path}}}
\newcommand{\removeextensions}{\mbox{{\tt remove}-{\tt exts}}}
\newcommand{\newshortestpath}{\mbox{{\tt new}-{\tt shortest}-{\tt path}}}
\newcommand{\update}{\mbox{{\tt update}}}

\begin{document}

\title{All-Pairs Shortest Paths in $O(n^2)$ time with high probability }

\author{\em Yuval Peres
\thanks{Microsoft Research, Redmond.
E-mail: {\tt peres@microsoft.com}.}
\and \em Dmitry Sotnikov
\thanks{School of Computer Science,
Tel Aviv University, Tel Aviv 69978, Israel. E-mail: {\tt
dmitrysot@gmail.com}.}
\and \em Benny Sudakov
\thanks{
Department of Mathematics, UCLA, Los Angeles, CA 90095, USA.
E-mail: {\tt bsudakov@math.ucla.edu}.
Research supported
in part by NSF CAREER award DMS-0812005 and by a USA-Israeli BSF grant.}
\and \em Uri Zwick
\thanks{School of Computer Science,
Tel Aviv University, Tel Aviv 69978, Israel. E-mail: {\tt
zwick@tau.ac.il}. Work supported by grant no.\ 2006261 of the United-States - Israel Binational Science Foundation (BSF).}
}

\date{}

\maketitle

\begin{abstract}\noindent
We present an all-pairs shortest path algorithm whose running time on a complete directed graph on~$n$ vertices whose edge weights are chosen independently and uniformly at random from $[0,1]$ is~$O(n^2)$, in expectation and with high probability. This resolves a long standing open problem. The algorithm is a variant of the dynamic all-pairs shortest paths algorithm of Demetrescu and Italiano. The analysis relies on a proof that the number of \emph{locally shortest paths} in such randomly weighted graphs is $O(n^2)$, in expectation and with high probability. We also present a dynamic version of the algorithm that recomputes all shortest paths after a random edge update in $O(\log^{2}n)$ expected time.
\end{abstract}

\section{Introduction}\label{S-intro}

The \emph{All-Pairs Shortest Paths} (APSP) problem is one of the most important, and most studied, algorithmic graph problems. Given a weighted directed graph $G=(V,E,c)$, on $|V|=n$ vertices and $|E|=m$ edges, where $c:E\to \reals^+$ is a length (or cost) function defined on its edges, we would like to compute 
the \emph{distances} between all pairs of vertices in the graph and a succinct representation of all \emph{shortest paths}. (The length of a path is the sum of the lengths of the edges participating in the path.)

The APSP problem can be solved in $O(mn+n^2\log n)$ worst-case time by running Dijkstra's algorithm from each vertex of the graph. (See Dijkstra \cite{Di59}, Fredman and Tarjan \cite{FrTa87}.) A slightly better running time of $O(mn+n^2\log\log n)$ was obtained by Pettie \cite{Pettie04}, building on techniques developed by Thorup \cite{Thorup99}. Karger, Koller and Phillips \cite{KaKoPh93} and McGeoch \cite{McGeoch95} developed algorithms that run in $O(m^*n+n^2\log n)$ time, where $m^*$ is the number of edges in the graph that are shortest paths. 

Demetrescu and Italiano \cite{DeIt04a,DeIt06} (see also Thorup \cite{Thorup04Dyn}) obtained a \emph{dynamic} APSP algorithm with an \emph{amortized} vertex update time of $\tilde{O}(n^2)$. Thorup \cite{Thorup05} obtained a dynamic algorithm with an $\tilde{O}(n^{2.75})$ \emph{worst-case} vertex update time. A \emph{vertex update} may insert, delete and change the weight of edges that touch a given vertex~$v$. An \emph{edge update} may only insert, delete or change the weight of a single edge. The algorithms of Demetrescu and Italiano \cite{DeIt04a,DeIt06} and Thorup \cite{Thorup04Dyn,Thorup05} can be used, of course, to perform edge updates, but the updates times may still be $\tilde{O}(n^2)$ and $\tilde{O}(n^{2.75})$, respectively.

Many researchers developed APSP algorithms that work well on \emph{random} instances, most notably complete directed graphs on $n$ vertices with random weights on their edges. The simplest such model, on which we focus in this paper, is the one in which all edge weights are drawn independently at random from the \emph{uniform} distribution on $[0,1]$. Hassin and Zemel \cite{HaZe85} and Frieze and Grimmett \cite{FrGr85} observed that, with very high probability, only the $O(\log n)$ cheapest edges emanating from each vertex participate in shortest paths. Thus, the APSP in this setting can be solved in $O(n^2\log n)$ expected time using the algorithms of Karger \etal \cite{KaKoPh93} and McGeoch \cite{McGeoch95}, or by simply selecting the $O(\log n)$ cheapest edges emanating from each vertex and then running Dijkstra's algorithm from each vertex. All these results actually hold in the more general setting in which edge weights are independent identically distributed random variables with a common cumulative distribution function~$F$ that satisfies $F(0)=0$ and $F'(0)$ exists and is strictly positive. (The uniform distribution on $[0,1]$ with $F(x)=x$, and the \emph{exponential} distribution, with $F(x)=1-{\rm e}^{-x}$ clearly satisfy these conditions.) Furthermore, the running time of these algorithms is $O(n^2\log n)$ with \emph{high probability}, i.e., probability that tends to~$1$ as $n$ tends to infinity, and not just in expectation.

Spira \cite{Spira73} obtained an APSP algorithm with an expected running time of $O(n^2\log^2 n)$ for complete directed graphs with edge weights drawn in an \emph{endpoint independent} manner. More specifically, for each vertex~$v$ a sequence of~$n$ positive numbers is chosen by an arbitrary deterministic or probabilistic process. These~$n$ numbers are then assigned to the~$n$ edges emanating from~$v$ in a \emph{random} order, with all $n!$ possible ordering being equally likely. Bloniarz \cite{Bloniarz83} presented an improved algorithm with an expected running time of $O(n^2 \log n \log^*n)$. Moffat and Takaoka \cite{MoTa87} and Mehlhorn and Priebe \cite{MePr97} improved the expected running time to $O(n^2\log n)$ and showed that it also holds with high probability.

Cooper \etal \cite{CoFrMePr00} obtained an APSP algorithm with an expected running time of $O(n^2\log n)$ in the \emph{vertex potential} model in which edge weights may be both positive and negative.

Meyer \cite{Meyer03}, Hagerup \cite{Hagerup06} and Goldberg \cite{Goldberg08} obtained \emph{Single-Source Shortest Paths} (SSSP) algorithms with an expected running time of $O(m)$.
The $m$-edge input graph may be arbitrary but its edge weights are assumed to be chosen at random from a common non-negative probability distribution. When the edge weights are independent, the running time of these algorithms is $O(m)$ with high probability.

Friedrich and Hebbinghaus \cite{FrHe08} presented an average case analysis of the dynamic APSP algorithm of Demetrescu and Italiano \cite{DeIt04a,DeIt06} on random \emph{undirected} graphs. The graphs in their analysis are chosen according to the $G(n,p)$ model, in which each edge of the complete graph is selected with probability~$p$, and edges of the random graph are given i.i.d.\ uniform random weights. They show that the expected edge update time is at most $O(n^{4/3+\epsilon})$, for any $\epsilon>0$. This bound is essentially tight when $p=1/n$, i.e., at the phase transition of the random graph, when the largest component is, with high probability, of size $\Theta(n^{2/3})$. When $p\ge (1+\epsilon')/n$, they show that the expected update time is $O(n^\epsilon/p)$, for every $\epsilon>0$.

Non-algorithmic aspects of distances and shortest paths in randomly weighted graphs were also a subject of intensive research in probability theory. We mention here only the results that are most relevant for us. Davis and Prieditis \cite{DaPr93} and Janson \cite{Janson99} showed that the expected distance of two vertices in a complete graph with random edge weights drawn independently from an \emph{exponential} distribution with mean~$1$ (i.e., $F(x)=1-{\rm e}^{-x}$) is exactly $H_{n-1}/(n-1)=(\ln n)/n + O(1/n)$, where $H_k=\sum_{i=1}^k \frac{1}{k}$ is the $k$-th Harmonic number. The probability that a given edge is a shortest path between its endpoints is also exactly $H_{n-1}/(n-1)$.
Exponential random variables are convenient to work with due to their \emph{memoryless} property. The same asymptotic results hold when the edges weights are chosen independently and uniformly from $[0,1]$.
In the exponential case, the tree of shortest paths from a given vertex has the same distribution as a \emph{random recursive tree} on~$n$ vertices obtained using the following simple process: Start with a root; add the remaining $n-1$ vertices, each time choosing the parent of the new vertex uniformly at random among the vertices that are already in the tree. The expected \emph{depth} of a vertex in such a tree, and hence the expected number of edges in a shortest path, is $\ln n+O(1)$. 
For further results regarding recursive trees and shortest paths, see
Devroye \cite{Devroye87}, Smythe and Mahmoud \cite{SmMa95} and Addario-Berry \etal~\cite{AdBrLu10}.

In their survey on the algorithmic theory of random graphs, Frieze and McDiarmid \cite{FrMc97} state the following open problem (Research Problem 22 on p.~28): ``Find a $o(n^2\log n)$ expected time algorithm for the all pairs
problem under a natural class of distributions, e.g., i.i.d.\ uniform on $[0,1]$.'' We solve this open problem by giving an $O(n^2)$ expected time algorithm for the problem, which is of course best possible. Furthermore, our
algorithm runs in $O(n^2)$ time with high probability and works for both directed and undirected versions of the all-pairs shortest paths problem.

Our $O(n^2)$-time APSP algorithm is a static version of the dynamic APSP algorithm of Demetrescu and Italiano \cite{DeIt04a,DeIt06} (see especially Section~3.4 of \cite{DeIt06})
with some modified data structures. The novel part of this paper is not the algorithm itself, but rather the probabilistic analysis that shows that it runs in $O(n^2)$ time, in expectation and with high probability.

We also obtain an $O(\log^{2}n)$ upper bound on the expected time needed to update all shortest paths following a random edge update, i.e., an update in which a random edge of the complete directed graph is selected and given a new random edge weight drawn uniformly at random from $[0,1]$.

The rest of the paper is organized as follows. In Section~\ref{S-DI} we sketch the static and dynamic versions of the algorithm of Demetrescu and Italiano \cite{DeIt04a,DeIt06} used in this paper. (Complete descriptions of these algorithms are given in Appendices~\ref{S-static} and~\ref{S-dynamic}.)
%
The crucial factor that determines the running time of these algorithms is the number of \emph{locally shortest paths} in the graph. A path is a \emph{locally} shortest path (\emph{LSP}) if the paths obtained by deleting its first and last edge, respectively, are \emph{shortest paths}.
In Section~\ref{S-DIST} we collect some known and some new results regarding the distances between vertices in randomly weighted graphs.
Using the results of Section~\ref{S-DIST}, we show in
Section~\ref{S-LSP}
that the \emph{expected} number of LSPs in a complete directed graph with independent uniformly distributed random weights is $O(n^2)$.
In Section~\ref{S-high} we show that the number of LSPs is~$O(n^2)$ \emph{with high probability}. Sections~\ref{S-LSP} and~\ref{S-high} are the main sections of the paper.
%
%
In Section~\ref{S-heap} we show that a fairly simple \emph{bucket} based priority queue, with a constant amortized update time, in conjunction with the fact that the number of LSPs is $O(n^2)$, in expectation and with high probability, yields the promised $O(n^2)$-time APSP algorithm.
In Section~\ref{S-update} we consider the expected time needed to perform \emph{random edge} updates. Interestingly, the arguments used in Sections~\ref{S-high} and~\ref{S-update} are related, as they both focus on the expected number of shortest paths that \emph{change} when a single edge is given a new random edge weight. (The link is the Efron-Stein inequality used in Section~\ref{S-high}.)
In Section~\ref{S-concl} we very briefly consider other random graph models.
In particular, our algorithm still runs in $O(n^2)$ expected time in the directed $G(n,p)$ model, in which each edge is present with probability~$p$, with independent uniformly distributed edge weights, at least when $p\gg (\ln n)/n$.
We end in Section~\ref{S-concl} with some concluding remarks and open problems.
%
%

\section{The algorithm of Demetrescu and Italiano}\label{S-DI}

Our $O(n^2)$ time bound, in expectation and with high probability, on the complexity of the solving the APSP problem on complete directed graphs with independent edge weight drawn uniformly from $[0,1]$, and the $O(\log^2 n)$ expected time bound on the complexity of performing a random edge update are both obtained using variants of the dynamic APSP algorithm of Demetrescu and Italiano \cite{DeIt04a,DeIt06}.

As our main result is the \emph{analysis} of these variants, and not the variants themselves, we begin by sketching the main features of the variants we use, mentioning only what the reader needs to know to understand our analysis. A complete description of the algorithms is given in Appendices~\ref{S-static} and \ref{S-dynamic}. (We believe that our variants are also of some interest, as they are not identical to the algorithms of \cite{DeIt04a,DeIt06}.)

Let $G=(V,E,c)$ be a weighted directed graph, where $c:E\to (0,\infty)$ is a cost function defined on its edges. (We use weights and costs interchangingly.)
For simplicity, we assume that all shortest paths in~$G$ are \emph{unique}. Under essentially all probabilistic models considered in this paper, this assumption holds with probability~$1$.
(Non-uniqueness of shortest paths can be dealt with as in \cite{DeIt04a}.)
We let $u\to v$ denote the edge $(u,v)\in E$, and let $u\SP v$ denote the (unique) shortest path from~$u$ to~$v$ in the graph, if they exist.


The key notion behind the algorithm of Demetrescu and Italiano \cite{DeIt04a} is the notion of \emph{locally shortest paths}.
\begin{definition}[Locally Shortest Paths]
A path is a \emph{locally shortest path} (or \emph{LSP}, for short) if the path obtained by
deleting its first edge, and the path obtained by deleting its last edge, are both shortest paths.
\end{definition}

More formally, if we let $u\to u'\SP v'\to v$ denote the path composed of the edge $u\to u'$, followed by the shortest path from $u'$ to~$v'$, and then by the edge $v'\to v$,
then $u\to u'\SP v'\to v$ is a locally shortest path if and only if $u\to u'\SP v'$ and $u'\SP v'\to v$ are both shortest paths. (If $u'=v'$, then $u'\SP v'$ is an empty path.) An edge is considered to be a locally shortest path. (Empty paths are considered to be shortest paths.)
%
A shortest path is of course also a locally shortest path.
A locally shortest path, however, is not necessarily a shortest path.

\subsection{A static version}\label{SS-static}

We begin by describing a static version of the algorithm of Demetrescu and Italiano \cite{DeIt04a,DeIt06}.
Let $G=(V,E,c)$ be a weighted directed graph. The algorithm constructs all shortest paths in~$G$ by essentially running Dijkstra's algorithm in parallel from all vertices, while only examining LSPs, as explained below.

For every $u,v\in V$, the algorithm maintains a number $dist[u,v]$ which is the length of the shortest path from~$u$ to~$v$ found so far. Initially $dist[u,v]$ is set to $c(u,v)$, if $(u,v)\in E$, or to $\infty$, otherwise. Each pair $(u,v)\in E$ is inserted into a heap (priority queue) $Q$, with $dist[u,v]$ serving as its key. The heap $Q$ holds all pairs of vertices $(u,v)$ such that at least one path from~$u$ to~$v$ in the graph was already discovered, but the shortest path from~$u$ to~$v$ was not yet declared.

In each iteration, the algorithm extracts a pair $(u,v)$ with the smallest key in~$Q$. As in Dijkstra's algorithm, $dist[u,v]$ is then the distance from~$u$ to~$v$ in~$G$. The algorithm then examines LSPs that \emph{extend} the shortest path $u\SP v$ and checks whether they are shorter than the currently best available paths between their endpoints.
(An extension of a path $\pi$ is a path obtained by adding an edge to its beginning or end.)
To efficiently find the LSPs that extend a shortest path $u\SP v$, the algorithm also maintains, in addition to $dist[u,v]$, the following information for every $u,v\in V$:

\smallskip
$\quad p[u,v]$ -- The \emph{second} vertex on the shortest path from~$u$ to~$v$ found so far.

$\quad q[u,v]$ -- The \emph{penultimate} (next to last) vertex on the shortest path from~$u$ to~$v$ found so far.

$\quad L[u,v]$ -- A list of vertices $w$ for which $w\to u\SP v$ is known to be a shortest path.

$\quad R[u,v]$ -- A list of vertices $w$ for which $u\SP v\to w$ is known to be a shortest path.

\smallskip
If no path from~$u$ to~$v$ was found yet, then $p[u,v]=q[u,v]=\NULL$.
The lists $L[u,v]$ and $R[u,v]$ specify the \emph{left} and \emph{right} extensions
of $u\SP v$ that are known to be shortest paths. Clearly $L[u,v]$ and $R[u,v]$ are non-empty only after the shortest path $u\SP v$ was identified by the algorithm.

Suppose that $u\to u'\SP v'\to v$, where $u'=p[u,v]$ and $v'=q[u,v]$, was just identified as a shortest path. For every $w\in L[u,v']$, $w\to u\SP v$ is an LSP. Similarly, for every $w\in R[u',v]$, $u\SP v\to w$ is an LSP. These paths are now examined by the algorithm. If, for example, a path $w\to u\SP v$ is found to be shorter then the currently available path from~$w$ to~$v$,
or is the first path found from~$w$ to~$v$, then $dist[w,v]$, $p[w,v]$ and $q[w,v]$ are updated accordingly and the key of $(w,v)$ in~$Q$ is decreased. (If $(w,v)$ is not already in~$Q$, it is inserted into~$Q$.)

This is the gist of the static version of the algorithm of Demetrescu and Italiano \cite{DeIt04a,DeIt06}, which, for concreteness, we refer to as algorithm \apsp. For a complete description and pseudo-code, see Appendix~\ref{S-static}.

As algorithm \apsp\ uses a priority queue, its running time depends on the characteristics of the priority queue used. For a specific implementation, we let $T_{ins}(n),T_{dec}(n)$ and $T_{ext}(n)$ denote  the (amortized) times of \emph{inserting} an element, \emph{decreasing} the key of a given element, and \emph{extracting} an element of minimum key from a priority queue containing at most~$n$ elements. We next claim:

\begin{theorem}\label{T-APSP-2}
If all edge weights are positive and all shortest paths are unique, then algorithm \apsp\ correctly finds all the shortest paths in the graph.
Algorithm \apsp\ runs in $O(n^2\cdot(T_{ins}(n^2)+T_{ext}(n^2)) + |\LSP|\cdot T_{dec}(n^2))$ time, where $|\LSP|$ is the number of LSPs in the graph, and uses only $O(n^2)$ space.
\end{theorem}


The proof of Theorem~\ref{T-APSP-2}, which is essentially identical to the correctness proof given by Demetrescu and Italiano \cite{DeIt04a,DeIt06}, can be found in Appendix~\ref{S-static}.

If we use the Fibonacci heaps data structure (Fredman and Tarjan \cite{FrTa87}) that supports \extractmin\ operations in $O(\log n)$ amortized time, and all other operations in $O(1)$, amortized time, where $n$ is the number of elements in the heap, we get a running time of $O(n^2\log n + |\LSP|)$.
There are, thus, two
hurdles on our way to getting an expected $O(n^2)$-time algorithm. First, we have to show that $|\LSP|$ is $O(n^2)$, under natural probability distributions, in expectation and with high probability. We do that in Sections~\ref{S-LSP} and~\ref{S-high}. Second, we have to find a faster way of implementing heaps. We do that in Section~\ref{S-heap} using a \emph{bucket} based implementation.

\subsection{A dynamic version}\label{SS-dynamic}

The static algorithm of the previous section examines all locally shortest paths in a graph, but (implicitly) maintains only those that are currently shortest. The dynamic algorithm, on the other hand, explicitly maintains all locally shortest paths, even if they are already known not to be shortest paths.

For every path $\pi$, we let $l[\pi]$ be the path obtained by deleting the last edge of~$\pi$, and $r[\pi]$ be the path obtained by deleting the first edge of~$\pi$.
A path $\pi$ is represented by keeping its total cost, its first and last edges, and pointers to its subpaths $l[\pi]$ and $r[\pi]$. The collection of all paths maintained by the algorithm is referred to as the \emph{path system}.

For every pair of vertices $u,v\in V$, the dynamic algorithm maintains a heap $P[u,v]$ that holds all the LSPs connecting~$u$ and~$v$ found so far. The key of each path is its cost. As in the static case, $dist[u,v]$ is the cost of the shortest path $\pi[u,v]$ from~$u$ to~$v$ found so far.

For every LSP $\pi$, the dynamic algorithm maintains four lists of left and right extensions of~$\pi$.
The lists $SL[\pi]$ and $SR[\pi]$ contain left and right extensions of~$\pi$ that are known to be shortest paths. The lists $L[\pi]$ and $R[\pi]$ contain extensions of~$\pi$ that are known to be LSPs.

Let $E'$ be a set of edges whose costs are changed by an update operation. (We are mostly be interested in the case in which $E'$ is composed of a single edge, but the description below is general.) The dynamic algorithm recomputes all shortest paths as follows. First all LSPs containing edges of~$E'$ are removed from the path system. (Note that each edge of~$E'$ is an LSP, and is thus contained in the path system. All LSPs containing edges of~$E'$ can be found by recursively following the extension lists of these edges.)

For every pair of vertices $u,v\in V$ such that the shortest path from~$u$ to~$v$ before the update passes through an edge of~$E'$, and was therefore removed from the path system, the algorithm finds the cheapest path in~$P[u,v]$, if at least one such path remains, and assigns it to $\pi[u,v]$. It then inserts the pair $(u,v)$ into a global heap~$Q$. The key of $(u,v)$ in~$Q$ is the cost of $\pi[u,v]$. Next, it recreates single-edge paths corresponding to the edges of~$E'$, with their new edge weights, and examines them.

The dynamic algorithm now starts to construct new shortest paths. In each iteration it extracts from~$Q$ a pair $(u,v)$ with the smallest key. As in the static case, the path $\pi[u,v]$ is then a shortest path from~$u$ to~$v$. LSP extensions of $\pi[u,v]$, obtained by combining $\pi[u,v]$ with paths that are already known to be shortest paths, are now generated. If such an extension is shorter than the currently shortest available path containing its endpoints~$u'$ and $v'$, then $\pi[u',v']$ and $dist[u',v']$ are updated accordingly, and $(u',v')$ is inserted into~$Q$ with the appropriate key. (If $(u',v')$ is already in~$Q$, its key is decreased.)

An important difference between the dynamic variant used in this paper and the dynamic algorithm of Demetrescu and Italiano \cite{DeIt04a,DeIt06} is that when a path $\pi$ stops being a shortest path, it, and all its extensions, are immediately removed from the path system. A similar dynamic variant was used by Friedrich and Hebbinghaus \cite{FrHe08}. The algorithm of Demetrescu and Italiano \cite{DeIt04a,DeIt06} keeps such paths as \emph{historical} and \emph{locally historical} paths. (See also Demetrescu \etal \cite{DeFaItTh06}.)

The most impressive feature of the dynamic algorithm of Demetrescu and Italiano \cite{DeIt04a,DeIt06} is that its update time is proportional to the number of shortest and locally shortest paths that are destroyed and/or created by the update operation. The algorithm does not spend time on shortest paths that remain unchanged.

Let $\SPP^-$ and $\LSP^-$ be the sets of shortest and locally shortest paths destroyed by an update operation.
Similarly, let $\SPP^+$ and $\LSP^+$ be the sets of shortest and locally shortest paths that are created (or recreated)
by an update operation.
Note that $\SPP^-$ and $\SPP^+$, and $\LSP^-$ and $\LSP^+$, are not necessarily disjoint, as paths passing through edges of~$E'$ are first destroyed, and removed from the path system, but may then be recreated. Let $\Lambda$ be an upper bound on the number of LSPs that connect any given pair of vertices before and after the update.

A complete description of the dynamic variant sketched here and its correctness proof are given in Appendix~\ref{S-dynamic}, where the following theorem is proved. ($\update(E',c')$ is the function that updates all shortest paths following a change in the costs of the edges of~$E'$.)

\begin{theorem}\label{T-dynamic} The running time of $\update(E',c')$ is
$$O(\; |\SPP^-|\cdot (T_{del}(\Lambda)+T_{min}(\Lambda)+T_{ins}(n^2)) +
|\SPP^+|\cdot T_{ext}(n^2) +
|\LSP^-|\cdot T_{del}(\Lambda) +
|\LSP^+|\cdot(T_{ins}(\Lambda)+T_{dec}(n^2) \;).$$
\end{theorem}

\vspace*{-1pt}
Here, $T_{ins}(n),T_{del}(n),T_{dec}(n),T_{ext}(n)$ and $T_{min}(n)$ are the (amortized) times of inserting, deleting, decreasing the key, extracting the element of minimum key, and finding the element of minimum key of a priority key containing at most~$n$ elements.

We show in Section~\ref{S-update} that for a random edge update we have $\Exp[\,|\SPP^-|\,],\Exp[\,|\SPP^+|\,]=O(\log n)$ and that $\Exp[\,|\LSP^-|\,],\Exp[\,|\LSP^+|\,]=O(\log^2 n)$. We also show that $\Lambda=O(\log n)$, with high probability. Using appropriate implementations of the priority queues, we get an expected edge update time of $O(\log^2 n)$.

\section{Distances in complete randomly weighted graphs}\label{S-DIST}

Let $K_n=(V,E)$ be a complete directed graph on~$n$ vertices
and let $a,b\in V$. We let $W(a,b)$ be the random weight attached to the edge $(a,b)$. We assume at first that $W(a,b)$ is an \emph{exponential} random variable with mean~$1$, i.e., $W(a,b)\sim EXP(1)$. Due to the memoryless property, dealing with exponentially distributed edge weights is easier than dealing directly with uniformly distributed edge weights. We later explain why all the results derived in this section for exponential edge weights also hold, asymptotically, for uniformly distributed edge weights.
All $n(n-1)$ random edge weights are assumed to be independent. (Self-loops, if present, may be ignored.) Let $D(a,b)$ be the \emph{distance} from~$a$ to~$b$ in the graph, i.e., the length (sum of weights) on the shortest path $a\SP b$ in the graph. (The shortest path $a\SP b$ is unique with probability~$1$.) Note that $D(a,b)$ is now also a random variable. For $k\in \{1,2,\ldots,n-1\}$, we let $D_k(a)$ be the distance from $a$ to the $k$-th closest vertex to~$a$.

Let $H_k = \sum_{k=1}^n \frac{1}{k}$ be the $k$-th \emph{Harmonic number}. It is known that
$H_n = \ln n + \gamma + O(\frac{1}{n}),$
where $\gamma=0.57721\ldots$ is Euler's constant.

The following five lemmas can be found in Janson \cite{Janson99}. (The expectation of $D(a,b)$, but not the variance, can also be found in Davis and Prieditis  \cite{DaPr93}). The lemmas in \cite{Janson99} are stated for undirected graphs, but it is easy to check that they also hold for directed graphs.

\begin{lemma}\label{L-Dk} Let $a\in V$ and $k\in \{1,2,\ldots,n-1\}$. Then, \[D_k(a) \;=\; \sum_{i=1}^k \frac{X_i}{i(n-i)}\;,\] where $X_1,X_2,\ldots,X_{k}$ are i.i.d.\ exponential variables with mean~$1$.
\end{lemma}

\begin{lemma}\label{L-Dab} Let $a\ne b\in V$. Then, \[D(a,b) \;=\; D_L(a) \;=\; \sum_{i=1}^L \frac{X_i}{i(n-i)}\;,\] where $X_1,X_2,\ldots,X_{n-1}$ are i.i.d.\ exponential variables with mean~$1$, and $L$ is chosen uniformly at random from $\{1,2,\ldots,n-1\}$.
\end{lemma}

\begin{lemma}\label{L-D} Let $a\ne b\in V$. Then,
$$\Exp[D(a,b)]\;=\;\frac{H_{n-1}}{n-1}\;=\; \frac{\ln n}{n} + O(\frac{1}{n})\quad,\quad
\Var[D(a,b)]\;=\;\frac{\pi^2}{2n^2} + o(\frac{1}{n^2}).$$
\end{lemma}


\begin{lemma}\label{L-max} For any constant $c>3$, we have $\displaystyle \Pr\left[\max_{a,b}D(a,b) \ge \frac{c\ln n}{n}\right] = O(n^{3-c}\log^2 n).$
\end{lemma}

\begin{lemma}\label{L-edge} Let $a\ne b\in V$. Then, the probability that the edge $a\to b$ is a shortest path is $\frac{H_{n-1}}{n-1}=\frac{\ln n}{n}+O(\frac{1}{n})$.
\end{lemma}

The next two lemmas are new and might be interesting in their own right. They are used in Section~\ref{S-high} to show that the running time of algorithm \apsp\ is $O(n^2)$ with high probability. The proof that the expected running time of \apsp\ is $O(n^2)$, given in Section~\ref{S-LSP}, does not rely on them.

\begin{lemma}\label{L-tail} Let $a\ne b\in V$. If $n^{-\alpha}<\alpha \leq 1/2$, then $\Pr[ D(a,b)>(1+12\alpha)\frac{\ln n}{n}] \le 5n^{-\alpha}$.
\end{lemma}

\begin{proof} Let $S_{k,\ell} = \sum_{i=k}^\ell \frac{X_i}{i(n-i)}$.
(We allow $k,\ell$ to be non-integral, in which case, we have $S_{k,\ell}=S_{\lceil k\rceil,\lfloor \ell \rfloor}$.)
By Lemma~\ref{L-Dab} we get that $D(a,b) = S_{1,L}$, where $L$ is uniformly distributed in $\{1,2,\ldots,n-1\}$. Let $m=n^{1-\alpha}$. We clearly have
\[ \Pr[D(a,b)>S_{1,n-m}] \;=\; \Pr[L>n-m] \;\le\; m/n \;=\; n^{-\alpha}.\eqno{(1)}\]
We now decompose
\[ S_{1,n-m} \;\le\; \frac{X_1}{n-1} + S_{2,m} + S_{m,n/2} + S_{n/2,n-m}. \eqno{(2)}\]
%
Now
\[ 
\Pr\left[\frac{X_1}{n-1}>2\alpha\frac{\ln n}{n}\right] \;\le\; \Pr[X_1>\alpha\ln n] \;=\; n^{-\alpha}.\eqno{(3)}\]
Let $Y=\sum_{i=2}^m\frac{X_i}{i}$. Using our assumption that $n^{-\alpha}<\alpha$ we get that
\[ 
S_{2,m} \;\le\; \frac{Y}{n-m} \;=\; \frac{Y}{(1-n^{-\alpha})n}
\;\le\;  \frac{Y}{(1-\alpha)n}\;\le\; (1+2\alpha)\frac{Y}{n}.\]
Now $\Exp[\ee^{\lambda X_i}]=(1-\lambda)^{-1}$, for $\lambda<1$, so
\[ 
\Exp\left[\ee^Y\right] \;=\;
\Exp\left[\ee^{\sum_{i=2}^m X_i/i}\right] \;=\; \prod_{i=2}^m \Exp\left[\ee^{X_i/i}\right] \;=\;
\prod_{i=2}^{m} \left(1-\frac{1}{i}\right)^{-1} \;=\; \prod_{i=2}^m \frac{i}{i-1} \;=\; m.\]
Therefore,
\[ \Pr\left[S_{2,m-1}\ge (1+2\alpha)\frac{\ln n}{n}\right] \;\le\; \Pr[Y\ge \ln n] \;=\; \Pr[\ee^Y\ge n]
\;\le\; \frac{\Exp[\ee^Y]}{n} \;=\; \frac{m}{n} \;=\; n^{-\alpha}. \eqno{(4)}\]
Let $Z=\sum_{i=m}^{n/2} \frac{X_i}{i}$. Clearly, $S_{m,n/2}\le \frac{2Z}{n}$. We again have
\[ \Exp[\ee^Z] \;=\; \prod_{i=m}^{n/2} (1-\frac{1}{i})^{-1} \;\le\; \frac{n}{m} \;=\; n^\alpha. \]
Therefore,
\[ 
\Pr\left[S_{m,n/2} \ge 4\alpha\frac{\ln n}{n}\right] \;\le\; \Pr[Z>2\alpha\ln n] \;=\; \Pr[ \ee^Z\ge n^{2\alpha}] \;\le\;
\frac{\Exp[\ee^Z]}{n^{2\alpha}} \;=\;
 n^{-\alpha}. \eqno{(5)} \]
As $S_{m,n/2}$ and $S_{n/2,n-m}$ have exactly the same distribution, we also get that
\[ 
\Pr\left[S_{n/2,n-m} \ge 4\alpha\frac{\ln n}{n}\right] \le n^{-\alpha} \eqno{(6)} \]
Using (1)-(6) together, we get that $\Pr\left[D(a,b)>(1+12\alpha)\frac{\ln n}{n}\right] \le 5n^{-\alpha}$, we required.
\end{proof}

No attempt to optimize the constants appearing the statement of Lemma~\ref{L-tail}. The condition $n^{-\alpha}<\alpha$ in the lemma is satisfied for any fixed $\alpha>0$, when $n$ is large enough. It also holds when, say, $\alpha=\alpha(n)=(\ln\ln n)/\ln n$.

The proof of our next lemma relies on the following large deviation theorem of Maurer \cite{Maurer03}.

\begin{theorem}[Maurer \cite{Maurer03}]\label{T-Maurer}
Let $Y_1,Y_2,\ldots,Y_n$ be non-negative independent random variables with finite first and second moments and let $S=\sum_{i=1}^n Y_i$. Let $t>0$. Then
$$\Pr\biggl[\Exp[S]-S\ge t\biggr] \le \exp\left(-\frac{t^2 }{2\sum_{i=1}^n \Exp[Y_i^2] } \right).$$
\end{theorem}

For a vertex $a\in V$ and $r>0$,
let $Ball(a,r)=\{b\in V \mid D(a,b)\le r\}$ be the \emph{ball} of radius~$r$ centered at~$a$. We next bound the probability that $Ball(a,\alpha \frac{\ln n}{n})$ is exceptionally large.

\begin{lemma}\label{L-alpha} For any $a\in V$, $\alpha\le 1$ and $c>0$ we have
\[ \Pr\left[ \left|Ball\left(a,\alpha\frac{\ln n}{n}\right)\right|> cn^{\alpha} \right] \;\le\;
\exp\left(-\frac{\ln^2 c}{60} \right).\]
\end{lemma}

\begin{proof}
Note that $|Ball(a,r)|> k$ if and only if $D_{k}(a)\le r$. By Lemma~\ref{L-Dk} we have $D_{k}=D_{k}(a) = \sum_{i=1}^{k} \frac{X_i}{i(n-i)}$, where $X_1,X_2,\ldots,X_{k}$ are i.i.d.\ exponential variables with mean~$1$.
Thus,
$$ \Exp[D_{k}] \;=\; \sum_{i=1}^{k}\frac{1}{i(n-i)} \;>\; \frac{1}{n}\sum_{i=1}^{k}\frac{1}{i}> \frac{\ln k}{n}.$$
As $\Exp[X_i^2]=2$, 
we have
$$
\sum_{i=1}^k \Exp\left[\left(\frac{X_i}{i(n-i)}\right)^2\right] = \sum_{i=1}^{k}\frac{2}{i^2(n-i)^2}\le \sum_{i=1}^{n-1}\frac{2}{i^2(n-i)^2}
\le 2\sum_{i=1}^{n/2}\frac{2}{i^2(n-i)^2}\le \frac{16}{n^2}\sum_{i=1}^{n/2}\frac{1}{i^2}\le
\frac{16}{n^2}\frac{\pi^2}{6}\le \frac{30}{n^2}.$$

With $k=cn^{\alpha}$ we get that $\Exp[D_k]>\frac{\alpha\ln n}{n}+\frac{\ln c}{n}$ and by Theorem~\ref{T-Maurer}, with $Y_i=\frac{X_i}{i(n-i)}$, we have
$$\Pr\left[ D_k\le \frac{\alpha\ln n}{n} \right] \;\le\;
\Pr\left[ \Exp[D_k] - D_k\ge \frac{\ln c}{n} \right] \;\le\;
\exp\left(- \frac{\left(\frac{\ln c}{n}\right)^2}{\frac{60}{n^2}} \right) \;=\; \exp\left( -\frac{\ln^2 c}{60} \right).$$
\end{proof}

As an immediate corollary, we get:

\begin{corollary}\label{C-alpha} For any $a\in V$, $\alpha\le 1$, $\epsilon>0$ and $c>0$ we have
\[ \Pr\left[ \left|Ball\left(a,\alpha\frac{\ln n}{n}\right)\right|>n^{\alpha+\epsilon} \right] \;=\; O(n^{-c})\;.\]
\end{corollary}




The results of this section were derived under the assumption that the edge weights are exponential. However, as explained in detail in the beginning of Section~2 of Janson \cite{Janson99}, the same results hold
asymptotically also for the uniform distribution. For the sake of completeness we show how to deduce from Lemma \ref{L-tail} and Corollary \ref{C-alpha} similar claims for uniform distributions.

Let $G$ be a complete directed graph on $n$ vertices with independent uniformly distributed edge weights $W(a,b)$ and let
$D(a,b)$ be the distance from $a$ to $b$ in this graph. Define $W'(a,b)=-\ln (1-W(a,b))$ and let $G'$ be a complete directed graph whose
edge weights are $W'(a,b)$. Denote by $D'(a,b)$ the distance from~$a$ to~$b$ in $G'$.
Note that all the edges of $G'$ have weights distributed as independent exponential random variables with mean $1$ and
that the correspondence between $G$ and $G'$ is a measure preserving transformation.
It is easy to check that
$ z \leq -\ln (1-z) \leq z+2z^2$ for all $0 \leq z \leq 1/2$.

Suppose that $G$ has the property that $D(a,b)>(1+12\alpha) \frac{\ln n}{n}$. Since $D'(a,b)\geq D(a,b)$, each such $G$ corresponds to a graph $G'$ which also has
$D'(a,b) >(1+12\alpha) \frac{\ln n}{n}$. Therefore by Lemma \ref{L-tail} the probability of this event is at most $5n^{-\alpha}$.
Suppose that $b$ is a vertex of $G$ satisfying $D(a,b) \leq \alpha \frac{\ln n}{n}$. Then, by the above inequality,  we have that
$D'(a,b) \leq \alpha \frac{\ln n}{n} +2\alpha^2 \frac{\ln^2 n}{n^2}=\alpha'\frac{\ln n}{n}$ with $\alpha'=\big(1+O(\frac{\ln n}{n})\big)\alpha$.
For any $\epsilon>0$, let $\epsilon'=\epsilon/2$. Then it is easy to check that $n^{\alpha'+\epsilon'} \leq n^{\alpha+\epsilon}$.
Therefore all $G$ in which $|Ball\left(a,\alpha\frac{\ln n}{n}\right)|>n^{\alpha+\epsilon}$ correspond to instances of $G'$
in which 
$|Ball\left(a,\alpha'\frac{\ln n}{n}\right)|>n^{\alpha'+\epsilon'}$.
By Corollary \ref{C-alpha} the probability of this event is at most $O(n^{-c})$ for any $c>0$.

\section{The expected number of locally shortest paths}\label{S-LSP}

Let $\LSP$ be the set of LSPs in $K_n$. Our goal in this section is to show that $\Exp[|\LSP|]=O(n^2)$.
This would follow immediately from the following lemmas.

\begin{lemma}\label{L-abc}  Let $a,b,c$ be three distinct vertices. The probability that $a\to b\to c$ is an LSP is $O(\frac{\ln^2 n}{n^2})$.
\end{lemma}




\begin{proof} The path $a\to b\to c$ is an LSP if and only if both $a\to b$ and $b\to c$ are shortest paths.
By Lemma~\ref{L-edge}, the probability that each one of the edges $a\to b$ and $b\to c$ is a shortest path is $\frac{\ln n}{n} + O(\frac{1}{n})$. Unfortunately, the events ``$a\to b$ is a shortest path'' and ``$b\to c$ is a shortest path'' are \emph{not} independent (and probably positively correlated). To circumvent that, let $V_1,V_2\subset V$ such that $V_1\cup V_2=V$, $a\in V_1$, $c\in V_2$, $V_1\cap V_2=\{b\}$ and $|V_1|,|V_2|\ge n/2$ be a fixed partition of the vertex set $V$. If $a\to b$ and $b\to c$ are shortest paths in~$G$, then $a\to b$ is clearly also a shortest path in $G[V_1]$, the subgraph of $G$ induced by~$V_1$, and $b\to c$ is also a shortest path in $G[V_2]$. These events are now independent, as the edge sets of $G[V_1]$ and $G[V_2]$ are disjoint. The probability that $a\to b\to c$ is an LSP is thus at most $(\frac{\ln(n/2)}{n/2}+O(\frac{1}{n}))^2$.
\end{proof}


\begin{lemma}\label{L-abcd} Let $a,b,c,d$ be four distinct vertices. 
The probability that $a\to b\SP c\to d$ is an LSP is $O(\frac{1}{n^2})$.
\end{lemma}

\begin{proof}
If $a\to b\SP c\to d$ is an LSP, then by definition
$$W(a,b) + D(b,c) = D(a,c) \quad,\quad  D(b,c) + W(c,d) = D(b,d).$$

If $a\to b\SP c\to d$ is an LSP, then $b\SP c$ does not pass through $a$ or $d$. (If, for example, $b\SP c$ passes through~$a$, then $a\SP c$ is a subpath of $b\SP c$, and $a\to b\SP c$ is therefore not a shortest path, contradicting the assumption that $a\to b\SP c\to d$ is an LSP.)
Thus, $D(b,c)=D_{a,d}(b,c)$, where $D_{a,d}(b,c)$ is the distance from~$b$ to~$c$ when~$a$ and~$d$ are removed from the graph.
We also clearly have $D(a,c)\le D_{b,d}(a,c)$ and $D(b,d)\le D_{a,c}(b,d)$.

Thus, if $a\to b\SP c\to d$ is an LSP, then
$$W(a,b) + D_{a,d}(b,c) \le D_{b,d}(a,c) \quad,\quad  D_{a,d}(b,c) + W(c,d) \le D_{a,c}(b,d),$$
or equivalently
$$W(a,b) \le D_{b,d}(a,c)-D_{a,d}(b,c) \quad,\quad W(c,d) \le D_{a,c}(b,d)-D_{a,d}(b,c).\eqno{(*)}$$
It is thus sufficient to bound the probability that $(*)$ happens.
For brevity, let $$X=D_{a,d}(b,c) \quad,\quad Y=D_{b,d}(a,c) \quad,\quad Z=D_{a,c}(b,d).$$
A crucial observation now is that $X,Y$ and~$Z$ do \emph{not} depend on $W(a,b)$ and $W(c,d)$. This follows from the fact that in each one of these distances one of~$a$ and~$b$, and one of~$c$ and~$d$, is removed from the graph.

We can thus choose the random weights of the edges in two stages. First we choose the random weights of all edges \emph{except} the two edges $a\to b$ and $c\to d$. The values of $X,Y$ and $Z$ are then already determined. We then choose $W(a,b)$ and $W(c,d)$, the random weights of the two remaining edges. As the choice of $W(a,b)$ and $W(c,d)$ is independent of all previous choices, and as $W(a,b)$ and $W(c,d)$ are independent and uniformly distributed in $[0,1]$, we get that
$$\Pr[(*)] \;=\; E\bigl[(Y-X)^+\cdot (Z-X)^+\bigr] \;\le\; E\bigl[|Y-X||Z-X|\bigr],$$
where $x^+=\max\{x,0\}$. (Note that we are not assuming here that $X,Y$ and $Z$ are independent. They are in fact dependent.)

We next note that each of $X,Y$ and $Z$ is the distance between two given vertices in a randomly weighted complete graph on $n-2$ vertices. Thus, $\Exp[X]=\Exp[Y]=\Exp[Z]$. By Lemma~\ref{L-D}, we have 
$$ \Var[X]=\Var[Y]=\Var[Z]\;=\;(1+o(1))\frac{\pi^2}{2n^2}.$$
Now,
$$\Pr[(*)] \;\le\; E\bigl[|Y-X||Z-X|\bigr] \;\le\; \frac{1}{2}(E\bigl[(Y-X)^2\bigr]+E\bigl[(Z-X)^2\bigr]),$$
using the trivial inequality $xy\le\frac{1}{2}(x^2+y^2)$. All that remains, therefore, is to bound $E\bigl[(Y-X)^2\bigr]$ and $E\bigl[(Z-X)^2\bigr]$. Let $\mu=\Exp[X]=\Exp[Y]$. Then,
$$\begin{array}{lcl}
E\bigl[(Y-X)^2\bigr] & = & E\bigl[ ((Y-\mu)-(X-\mu))^2 \bigr] \\[3pt]
           & \le & 2(E\bigl[(Y-\mu)^2\bigr]+E\bigl[(X-\mu)^2\bigr])\\[3pt]
           & = & 2(\Var[Y]+\Var[X]) \\[3pt]
           & = & (1+o(1))\frac{\pi^2}{n^2},
\end{array}$$
using the inequality $(x-y)^2\le 2(x^2+y^2)$. Exactly the same bound applies to $E\bigl[(Z-X)^2]$. Putting everything together, we get that $\Pr[(*)]\le (1+o(1))\frac{\pi^2}{n^2}$.
\end{proof}

\begin{theorem}\label{T-LSP} $\Exp[|\LSP|]=\Theta(n^2)$
\end{theorem}

\begin{proof} The number of LSPs of length~$1$ is exactly $n(n-1)$. (Every edge is an LSP of length~$1$.) By Lemma~\ref{L-abc}, the expected number of LSPs of length~$2$ is $O(n^3\cdot \frac{\ln^2 n}{n^2})=O(n\ln^2 n)$. By Lemma~\ref{L-abcd}, the expected number of LSPs of length greater than two is $O(n^4\cdot \frac{1}{n^2})=O(n^2)$.
\end{proof}


Experiments that we have done seem to suggest that $\Exp[|\LSP|]$ is very close to $(\frac{\pi^2}{6}+1)n^2\simeq 2.64n^2$.

The results of this section were stated and proved for directed graphs. It is easy to check, however,
that our methods can be also used to provide an all pairs shortest paths algorithm with a quadratic running time
for the complete \emph{undirected} graphs on~$n$ vertices with uniform edge weights.



\section{High probability bound on the number of locally shortest paths}\label{S-high}


Our goal in this section is to show that the number of LSPs is $O(n^2)$ asymptotically almost surely (a.a.s), i.e., that there exists a constant~$c$ such that $\Pr[|\LSP|<cn^2]\to 1$, as $n\to\infty$.

Let $E^*$ be the set of edges that are shortest paths. Let $\Delta$ be the maximum outdegree in the subgraph $G^*=(V,E^*)$. (McGeoch \cite{McGeoch95} refers to $G^*=(V,E^*)$ as the \emph{essential} subgraph.)
We first show that $\Delta=O(\log n)$, with very high probability.

\begin{lemma}\label{L-Delta} For every $c>6$, we have $\Pr[\Delta>c\ln n]=O(n^{1-c/6})$.
\end{lemma}

\begin{proof} Let $G'=(V,E')$ be the subgraph of~$G$ composed of all edges of weight at most $\frac{c}{2}\frac{\ln n}{n}$, and let~$\Delta'$ be the maximum outdegree in~$G'$. The outdegree of each vertex in~$G'$ is binomially distributed with parameters~$n$ and $\frac{c}{2}\frac{\ln n}{n}$.
A special case of Chernoff bound (see, e.g., \cite{MiUp05}, p.~64) states that if $X$ is a binomial variable with $\mu=\Exp[X]$, then $\Pr[X\ge 2\mu]\le {\rm e}^{-\mu/3}$. Thus, the probability that the degree of a given vertex exceeds $c\ln n$ is at most $n^{-c/6}$. Thus $\Pr[\Delta'>c\ln n]\le n^{1-c/6}$.
Now, $\Delta>\Delta'$ only if at least one distance in~$G$ is greater than $\frac{c}{2}\frac{\ln n}{n}$.
By Lemma~\ref{L-max}, the probability that this happens is at most $O(n^{3-c/2}\log^2 n)$. For $c>6$ we have $1-c/6>3-c/2$.
\end{proof}

The following lemma is trivial and can also be found in Demetrescu and Italiano \cite{DeIt04a}.

\begin{lemma}\label{L-simple} If all shortest paths are unique, then $|\LSP|\le \Delta n^2$.
\end{lemma}

\begin{proof} Every LSP is obtained by appending an edge which is itself a shortest path to some shortest path.
The number of shortest path in a graph is at most $n^2$ (assuming uniqueness) and each one of these shortest path can be extended by at most~$\Delta$ edges.
\end{proof}

Note that Lemmas~\ref{L-Delta} and~\ref{L-simple} imply that the number of LSPs is $O(n^2\log n)$ with high probability. To improve this bound to $O(n^2)$ we need to work harder.

\begin{definition}[$\beta$-short paths]
Let $\beta>0$ be a (small) constant. 
We say that a shortest path $\pi$ is \emph{$\beta$-short} if and only if its length is at most $(1+\beta)\frac{\ln n}{n}$, and \emph{$\beta$-long}, otherwise. Similarly, we say that an LSP~$\pi$ is \emph{$\beta$-short} if both shortest paths~$l[\pi]$ and~$r[\pi]$ obtained by removing its first edge and last edge are short, and \emph{$\beta$-long}, otherwise. Let $\SPP^S$, $\SPP^L$, $\LSP^S$, $\LSP^L$ be the sets of $\beta$-short and $\beta$-long shortest and locally shortest paths. (Note that these sets depend on the parameter~$\beta$.)
\end{definition}

Clearly, $|\LSP|=|\LSP^L|+|\LSP^S|$. We estimate separately the number of $\beta$-long LSPs and the number of $\beta$-short LSPs.
We begin by bounding the number of $\beta$-long shortest paths and locally shortest paths.




\begin{lemma}\label{L-long-SP} For every $\beta>0$, we have $\Exp[|\SPP^L|]=O(n^{2-\beta/12})$.
\end{lemma}

\begin{proof} By Lemma~\ref{L-tail}, with $\alpha=\beta/12$, we get that for any $a\ne b\in V$ we have
\[\Pr[D(a,b)\ge (1+\beta)\frac{\ln n}{n}] \;=\; O(n^{-\beta/12})\;.\] The lemma follows by the linearity of expectation.
\end{proof}

\begin{lemma}\label{L-long-LSP} For every $\beta>0$, we have $\Exp[|\LSP^L|]=O(n^{2-\beta/12}\ln n)$.
\end{lemma}

\begin{proof} Using the same argument used in the proof of Lemma~\ref{L-simple}, we get that $|\LSP^L|\le\Delta |\SPP^L|$. By Lemma~\ref{L-Delta} we get
\[ \Exp[ |\LSP^L| ] \;\le\; \Exp[ \Delta |\SPP^L| ] \;\le\; c\ln n \cdot \Exp[ |\SPP^L| ] + n^{1-c/6} n^3\;.\]
Letting $c=12$ and using Lemma~\ref{L-long-SP} we get that $\Exp[|\LSP^L|]=O(n^{2-\beta/12}\ln n)$, as required.
\end{proof}

\begin{lemma}\label{L-Pr-LSP_L} For every $\beta>0$ we have $\Pr[|\LSP^L|\ge n^2]=O(n^{-\beta/12}\ln n)$.
\end{lemma}

\begin{proof} Follows from Lemma~\ref{L-long-LSP} using Markov's inequality.
\end{proof}

We next show that $|\LSP^S|=O(n^2)$ with high probability. To do that we use the \emph{Efron-Stein inequality} (see, e.g., Boucheron \etal~\cite{BoLuBo03}) to bound $\Var[|\LSP^S|]$.

\begin{theorem}[Efron-Stein inequality]\label{T-Efron}
Let $Z=f(X_1,\ldots,X_m)$, where $X_1,X_2,\ldots,X_m$ are independent random variables. For any $1\le i\le m$, let $X'_i$ be a random variable with the same distribution as $X_i$ but independent from $X_1,X_2,\ldots,X_m$, and let $Z'_i=f(X_1,\ldots,X'_i,\ldots,X_m)$. Then,
$$\Var[Z] \;\le\; \frac12 \sum_{i=1}^m E\left[ \bigl(Z-Z'_i\bigr)^2\right].$$
\end{theorem}

In our case, we have $m=n(n-1)$, $X_1,X_2,\ldots,X_m$ are the random edge weights, and $Z=|\LSP^S|$. For every edge $e$, we need to compute the second moment of the random variable $|\LSP^S_{e,0}|-|\LSP^S_{e,1}|$, where $\LSP^S_{e,0}$ and $\LSP^S_{e,1}$ are the sets of $\beta$-short LSPs when all edges other than~$e$ are assigned the \emph{same} random edge weights, while $e$ is assigned two independent edge weights. Due to symmetry, the second moment of $|\LSP^S_{e,0}|-|\LSP^S_{e,1}|$ does not depend on~$e$. For brevity, we write $\LSP^S_{0}$ and $\LSP^S_{1}$, instead of $\LSP^S_{e,0}$ and $\LSP^S_{e,1}$, when the $e$ is clear from the context.
We similarly define $\SPP^S_0$ and $\SPP^S_1$ to be the corresponding sets of $\beta$-short shortest paths.


If $A$ and~$B$ are two sets, then $||A|-|B||\le |A\oplus B|$, where $A\oplus B=(A\smallsetminus B) \cup (B\smallsetminus A)$ is the \emph{symmetric difference} of the two sets. We thus focus our attention on $\LSP^S_0\oplus \LSP^S_1$. We begin by looking at $\SPP^S_0\oplus \SPP^S_1$.
Let $\SPP^S_0(e)$ and $\SPP^S_1(e)$ be the set of $\beta$-short shortest paths that \emph{pass} through~$e$ with the two choices of the weight of~$e$.

\begin{lemma}\label{L-SPe} For every edge $e$ we have $|\SPP^S_0\oplus \SPP^S_1|\le 2(|\SPP^S_0(e)|+|\SPP^S_1(e)|)$.
\end{lemma}

\begin{proof} Let $c_0(e)$ and $c_1(e)$ be the two costs of~$e$. Suppose at first that $c_0(e)<c_1(e)$. A $\beta$-short shortest path that stops being $\beta$-short shortest path when the cost of~$e$ is increased from~$c_0(e)$ to~$c_1(e)$ must pass through~$e$. Thus, $\SPP^S_0\smallsetminus\SPP^S_1\subseteq \SPP^S_0(e)$ and hence $|\SPP^S_0\smallsetminus\SPP^S_1|\le |\SPP^S_0(e)|$. The only paths in $\SPP^S_1\smallsetminus\SPP^S_0$ are paths that replace paths from $\SPP^S_0\smallsetminus\SPP^S_1$. Thus, we also have $|\SPP^S_1\smallsetminus \SPP^S_0|\le |\SPP^S_0(e)|$. Under the assumption $c_0(e)<c_1(e)$ we thus get
$|\SPP^S_0\oplus \SPP^S_1|\le 2|\SPP^S_0(e)|$. If $c_0(e)>c_1(e)$, we similarly get that $|\SPP^S_0\oplus \SPP^S_1|\le 2|\SPP^S_1(e)|$. In both cases we have $|\SPP^S_0\oplus \SPP^S_1|\le 2(|\SPP^S_0(e)|+|\SPP^S_1(e)|)$.
\end{proof}




%
We next estimate $|\SPP^S_0(e)|$ and $|\SPP^S_1(e)|$. As they both have the same distribution, we omit the subscript.


\begin{lemma}\label{L-zero} For every $\beta>0$, we have $\Pr[|\SPP^S(e)|>0]=O(\frac{\ln n}{n})$.
\end{lemma}

\begin{proof}
The set $\SPP^S(e)$ is non-empty only if $e$ is a shortest path between its endpoints, which by Lemma~\ref{L-edge} only happens with probability $\frac{\ln n}{n}+O(\frac{1}{n})$.
\end{proof}

Our next goal is to show that $|\SPP^S(e)|={O}(n^{1+\beta'})$, with high probability, for any $\beta'>\beta$.





\begin{lemma}\label{L-SPP-e} For every $\beta>0$, and every $\beta'>\beta$, we have $\Pr[ |\SPP^S(e)| >
n^{1+\beta'}] = O(n^{-c})$, for every $c>0$.
\end{lemma}

\begin{proof} Let $e=a\to b$ be a fixed edge.
Let $C$ be the set of pairs $(u,v)$ such that $u\SP a\to b\SP v$ is a shortest path of length at most $\frac{(1+\beta)\ln n}{n}$. Clearly $|\SPP^S(e)|=|C|$.
For a fixed integer $r$, and $1\le i\le r$, let $$A_i=\left\{u\in V \,\left|\, D(u,a)\le \frac{i(1+\beta)}{r}\frac{\ln n}{n}\right.\right\} \quad,\quad B_i=\left\{v\in V \,\left|\, D(b,v)\le \frac{i(1+\beta)}{r}\frac{\ln n}{n}\right.\right\}$$ be the sets of vertices of distances at most $\frac{i(1+\beta)}{r}\frac{\ln n}{n}$ to~$a$ and from~$b$, respectively.
Note that $B_i=Ball\left(b,\frac{i(1+\beta)}{r}\frac{\ln n}{n}\right)$, while $A_i=Ball\left(a,\frac{i(1+\beta)}{r}\frac{\ln n}{n}\right)$, in the graph in which all edge directions are reversed. Clearly
$$C \;\subseteq\; \bigcup_{i=1}^{r} A_i\times B_{r+1-i}\;.$$
By Corollary~\ref{C-alpha}, we have $|A_i|,|B_i|\le n^{(1+\beta)i/r+\epsilon}$, for every~$i$, with
a probability of at least $1-O(r n^{-c})$, for every $c>0$.
It thus follows that $|C|\le r n^{(1+\beta)(1+1/r)+2\epsilon}$, again with this very high probability. Letting~$r$ sufficiently large and~$\epsilon$ sufficiently small, we get the claim of the lemma.
\end{proof}

Lemmas~\ref{L-zero} and~\ref{L-SPP-e} allow us to bound $\Exp[|\SPP^S(e)|^2]$.

\begin{lemma}\label{L-E-SPP-2} For every $\beta>0$, and every $\beta'>\beta$ we have $\Exp[|\SPP^S(e)|^2]=O(n^{2(1+\beta')-1}\ln n)$.
\end{lemma}

\begin{proof} For succinctness, let $X=|\SPP^S(e)|$ and $a=n^{1+\beta'}$.
We always have $X^2\le n^4$. Using Lemma~\ref{L-SPP-e} with $c=4$, we have
$$
\Exp[X^2] \;\le\; \Pr[0<X\le a]\cdot a^2 + \Pr[x>a]\cdot n^4 \;=\; O\left(\frac{\ln n}{n}\cdot n^{2(1+\beta')} + n^{-4}\cdot n^4\right) \;=\; O\left(n^{2(1+\beta')-1}\ln n\right).$$
\end{proof}

We can finally get back to estimating $\LSP^S_0\oplus \LSP^S_1$. Let $\Delta_0$ and $\Delta_1$ be the maximum outdegrees in the essential graph, i.e., the subgraph composed of the edges that are shortest paths, under the two independent choices of the weight of~$e$. Let $\Delta=\max\{\Delta_0,\Delta_1\}$. By Lemma~\ref{L-Delta} we have
$\Pr[\Delta>c\ln n]=O(n^{1-c/6})$, for every $c>6$.

\begin{lemma}\label{L-oplus} For every $\beta>0$ we have $|\LSP^S_0\oplus \LSP^S_1|\le 2\Delta\cdot |\SPP^S_0\oplus \SPP^S_1|$.
\end{lemma}

\begin{proof} Suppose that $\pi\in \LSP^S_0\smallsetminus \LSP^S_1$. Then either $l[\pi]\in \SPP^S_0\smallsetminus \SPP^S_1$ or $r[\pi]\in \SPP^S_0\smallsetminus \SPP^S_1$. Each shortest path in $\SPP^S_0$ has at most $\Delta_0$ pre-extensions and at most $\Delta_0$ post-extensions that are locally shortest paths. Thus,
$|\LSP^S_0\smallsetminus \LSP^S_1|\le 2\Delta_0 |\SPP^S_0\smallsetminus \SPP^S_1|$. Similarly, $|\LSP^S_1\smallsetminus \LSP^S_0|\le 2\Delta_1 |\SPP^S_1\smallsetminus \SPP^S_0|$, and the lemma follows.
\end{proof}

\begin{lemma} For every $\beta>0$, every $\beta'>\beta$, and every edge~$e$ we have \[ \Exp\left[\bigl||\LSP^S_1|-|\LSP^S_0|\bigr|^2\right] =
O\left(n^{2(1+\beta')-1}\ln^3 n\right).\]
\end{lemma}

\begin{proof} By Lemmas~\ref{L-SPe} and~\ref{L-oplus} we have
$$\bigl||\LSP^S_1|-|\LSP^S_0|\bigr|^2 \;\le\; \bigl| \LSP^S_0 \oplus \LSP^S_1 \bigr|^2 \;\le\;
4\Delta^2 \bigl| \SPP^S_0 \oplus \SPP^S_1 \bigr|^2$$ $$ \;\le\;
16\Delta^2 \bigl( |\SPP^S_0(e)| + |\SPP^S_1(e)| \bigr)^2 \;\le\; 32\Delta^2 ( |\SPP^S_0(e)|^2 + |\SPP^S_1(e)|^2).$$
Since both $|\SPP^S_0(e)|, |\SPP^S_1(e)| \leq n^2$, from Lemma ~\ref{L-Delta} with $c=24$ we have
\begin{eqnarray*}
\Exp\left[\bigl||\LSP^S_1|-|\LSP^S_0|\bigr|^2\right] &\leq& O\left( \ln ^2 n\cdot \Exp[|\SPP^S(e)|^2]+ \Pr[\Delta>24\ln n] \cdot n^4\right)\\
&=&
O\left(\ln ^2 n\cdot \Exp[|\SPP^S(e)|^2]+n^{-3} \cdot n^4\right).
\end{eqnarray*}
The claim now follows from Lemma~\ref{L-E-SPP-2}.
\end{proof}

Using the Efron-Stein inequality (Theorem~\ref{T-Efron}) we thus get:

\begin{lemma}\label{L-Var} For every $\beta>0$ and every $\beta'>\beta$ we have $\Var\left[|\LSP^S|\right] = O\left(n^{2(1+\beta')+1} \ln^3 n\right)$.
\end{lemma}

\begin{theorem}\label{T-hp} There is a constant $c$ such that $\Pr[|\LSP|\ge cn^2]=O(n^{-1/26})$.
\end{theorem}

\begin{proof} Let $\beta=\frac{12}{25}$. By Lemma~\ref{L-Pr-LSP_L} we get that
\[\Pr[|\LSP^L|\ge n^2] \;=\; O(n^{-\beta/12}\ln n) \;=\; O(n^{-1/25}\ln n) \;.\eqno{(7)}\]
Let $\beta'=\beta+\epsilon$, where $\epsilon>0$ is tiny.
By Lemma~\ref{L-Var} we get that
\[\Var\left[|\LSP^S|\right] \;=\; O\left(n^{2(1+\beta')+1} \ln^3 n\right) \;=\; O(n^{99/25+2\epsilon}\ln^3 n)
\;=\; O(n^{99/25+3\epsilon})
.\]
By Theorem~\ref{T-LSP} we have $\Exp[|\LSP|]=\Theta(n^2)$. By Lemma~\ref{L-long-LSP}, we have $\Exp[|\LSP^L|=o(n^2)$. As
$\Exp[|\LSP|] = \Exp[|\LSP^S|]+\Exp[|\LSP^L|]$, we get that $\Exp[|\LSP^S|]=\Theta(n^2)$.

By Chebyshev's inequality (see, e.g., \cite{MiUp05}), for every random variable~$X$ we have \[\Pr[X\ge 2\Exp[X]] \;\le\; \Pr[|X-\Exp[X]|\ge \Exp[X]]\;\le\; \frac{\Var[X]}{\Exp[X]^2}\;.\]
For $X=|\LSP^S|$, and using the facts that $\Exp\left[|\LSP^S|\right]=\Theta(n^2)$ and
$\Var[|\LSP^S|]=O(n^{99/25+3\epsilon})$, we thus get
\[ \Pr\left[|\LSP^S|\ge 2\Exp[|\LSP^S|]\right] \;\le\; \frac{\Var\left[|\LSP^S|\right]}{\Exp\left[|\LSP^S|\right]^2} \;=\; O(n^{-1/25+3\epsilon})\;. \eqno{(8)}
\]
As $|\LSP|=|\LSP^S|+|\LSP^L|$,
combining (7) and (8) and choosing $\epsilon$ small enough, we get the claim of the Theorem.
\end{proof}

We believe that for every $a>0$ there exists~$c$ such that $\Pr\left[|\LSP|\ge cn^2\right]=O(n^{-a})$.
Proving, or disproving, this claim would require new techniques.

\section{An $O(n^2)$-time implementation}\label{S-heap}

In this section we describe an implementation of the algorithm of Section~\ref{SS-static} (and Appendix~\ref{S-static}) that runs in $O(n^2)$ time in expectation and with high probability. This is done using
a simple observation of Dinic \cite{Dinic78} and a simple bucket-based priority queue implementation that goes back to Dial \cite{Dial69}.

Let $\delta=\min_{(u,v)\in E} c(u,v)$ be the minimal edge weight in the graph. We claim that algorithm \apsp\ of Section~\ref{SS-static} remains correct if instead of requiring that the pair $(u,v)$ extracted from the heap~$Q$ is a pair with minimal $dist(u,v)$, we only require that
$dist(u,v)<dist(u',v')+\delta$ for every other pair $(u',v')$ in~$Q$. The proof is a simple modification of the proof of Theorem~\ref{T-APSP-2} given in Appendix~\ref{S-static}. This observation, in the context of Dijkstra's algorithm, dates back to Dinic \cite{Dinic78}. Along with many more ideas, this observation forms the basis for the linear \emph{worst-case} time single-source shortest paths algorithm for undirected graphs obtained by Thorup \cite{Thorup99}. It is also used by Hagerup \cite{Hagerup06} to obtain a simple linear \emph{expected} time algorithm for single source shortest paths, simplifying results of Meyer \cite{Meyer03} and Goldberg \cite{Goldberg08}.

In our setting, edge weights are drawn independently and uniformly at random from $[0,1]$. The probability that the minimal edge weight is smaller than $n^{-2.5}$ is clearly at most $n^{-0.5}$. If this unlikely event happens, we simply use an $O(n^2\log n)$ time implementation based on Fibonacci heaps.
This only contributes $o(n^2)$ to the expected running time of the algorithm.

We assume now that $\delta \ge n^{-2.5}$. For every $u,v\in V$, we let $dist'(u,v)=\lfloor dist(u,v)/\delta \rfloor$ and use $dist'(u,v)$, instead of $dist(u,v)$, as the key of $(u,v)$ in~$Q$.

We implement the heap~$Q$ as follows. (There are many possible variants. We describe the one that seems to be the most natural.) We use $L=n^2$ buckets $B_1,B_2,\ldots,B_L$. Bucket $B_i$, for $i<L$, is a linked list holding pairs $(u,v)$ for which $dist'(u,v)=i$. Bucket $B_L$ is a special \emph{leftover} bucket that holds all pairs $(u,v)$ for which $dist'(u,v)\ge L$. It is again implemented as a linked list. We also maintain the index~$k$ of the bucket from which the last minimal pair was extracted.

The implementation of a \heapinsert\ operation is trivial.
To insert a pair $(u,v)$ into~$Q$, we simply add $(u,v)$ to $B_i$, where $i=\min\{dist'(u,v),L\}$.

A \decreasekey\ operation is also simple. We simply remove $(u,v)$ from its current bucket and move it to the appropriate bucket. (Each pair has a pointer to its position in its current bucket, so these operations take constant time.)

An \extractmin\ operation is implemented as follows. We sequentially scan the buckets, starting from~$B_k$, until we find the first non-empty bucket. If the index of this bucket is less than~$L$, we return an arbitrary element from this bucket and update~$k$ if necessary. If the first non-empty bucket is $B_L$, the leftover bucket, we insert all the elements currently in $B_L$ into a comparison-based heap and use it to process all subsequent heap operations. (We show below that in our setting, we would very rarely encounter this case.)

This implementation of the \extractmin\ operation is correct as the priority queue that we need to maintain is \emph{monotone}, in the sense that the minimal key contained in the priority queue never decreases. This follows immediately from then fact that keys of new pairs inserted into~$Q$, or decreased keys of existing pairs in~$Q$ are always larger than the key of the last extracted pair.

The total time spent on implementing all heap operations, until all buckets $B_1,\ldots,B_{L-1}$ are empty, is clearly $O(N+L)$, where $N$ is the number of heap operations performed. By Theorem~\ref{T-APSP-2} we have $N=O(|\LSP|+n^2)$. By Theorem~\ref{T-LSP} we have $\Exp[|\LSP|]=O(n^2)$.  By Theorem~\ref{T-hp}, there is constant~$c$ such that $\Pr[|\LSP|\ge cn^2] = O(n^{-1/60})$. As $L=n^2$, the number of operations here is $O(n^2)$, both in expectation and with high probability.

All that remains, therefore, is to show that the probability that $B_L$ will be the only non-empty bucket is tiny. Note that this happens if and only if there is a pair $u,v\in V$ for which $D(u,v)\ge L\delta\ge n^{-0.5}$.
By Lemma~\ref{L-max}, this probability is $O(n^{-c})$ for every $c>0$. If this extremely unlikely event happens, the running time is only increased to $O(n^2\log n)$, which has a negligible effect on the expected running time of the whole algorithm. We have thus obtained:

\begin{theorem} The expected running time of algorithm \apsp, when implemented using a bucket-based priority queue, and when run on a complete directed graph with edge weights selected uniformly at random from $[0,1]$ is $O(n^2)$. Furthermore, there is a constant $c>0$ such that the probability that the running time of the algorithm exceeds $cn^2$ is $O(n^{-1/60})$.
\end{theorem}

\section{Polylogarithmic update times}\label{S-update}

In this section we consider the expected time needed to update all shortest paths following a \emph{random edge update}, i.e., an update operation that chooses a random edge~$e$ of the complete directed graph, uniformly at random, and assigns it a new random weight, independent of all previous weights chosen, drawn uniformly at random from $[0,1]$.

Recall that $\SPP^-$ and $\LSP^-$ are the sets of shortest and locally shortest paths destroyed by an update operation, and that $\SPP^+$ and $\LSP^+$ are the sets of shortest and locally shortest paths that are created (or recreated) by an update operation. More specifically, we have
\[\begin{array}{ccc}
\SPP^- &=& \SPP_0(e) \;\cup\; (\SPP_0\smallsetminus \SPP_1)\;,\\
\SPP^+ &=& \SPP_1(e) \;\cup\; (\SPP_1\smallsetminus \SPP_0)\;,\\
\LSP^- &=& \LSP_0(e) \;\cup\; (\LSP_0\smallsetminus \LSP_1)\;,\\
\LSP^+ &=& \LSP_1(e) \;\cup\; (\LSP_1\smallsetminus \LSP_0)\;,
\end{array}\]
where, as in Section~\ref{S-high}, $\SPP_0$ and $\SPP_1$ are the sets of shortest paths before and after the update of~$e$, and $\SPP_0(e)$ and $\SPP_1(e)$ are the sets of shortest paths, before and after the update, that pass through~$e$. The sets $\LSP_0$, $\LSP_1$, $\LSP_0(e)$ and $\LSP_1(e)$, are the corresponding sets of locally shortest paths.

Our main goal is to bound the expected sizes of the sets $\SPP^-$, $\SPP^+$, $\LSP^-$ and $\LSP^+$. This, in conjunction with Theorem~\ref{T-dynamic}, would supply an upper bound on the expected update time.
By symmetry, it is easy to see that $\Exp[|\SPP^-|]=\Exp[|\SPP^+|]$ and $\Exp[|\LSP^-|]=\Exp[|\LSP^+|]$. We can thus concentrate on estimating $\Exp[|\SPP^-|]$ and $\Exp[|\LSP^-|]$.

Let $e$ be the random edge updated by a random edge update operation.
For every $u,v\in V$, let~$\pi_0[u,v]$ and~$\pi_1[u,v]$ be the shortest path from~$u$ to~$v$ before and after the update.
%
%
%
Let $B_i=\{(u,v) \mid e\in \pi_i[u,v]\}$, for $i\in\{0,1\}$, be the set of pairs of vertices
connected, before and after the update, by a shortest path passing through~$e$. (Note that $|B_i|=|\SPP_i(e)|$, for $i\in\{0,1\}$.)
It is easy to see that $\pi_0[u,v]\in \SPP^-$ if and only if $e\in \pi_0[u,v]$ or $e\in \pi_1[u,v]$. Thus, $\SPP^-=\{ \pi_0[u,v] \mid (u,v)\in B_0\cup B_1\}$ and similarly
$\SPP^+=\{ \pi_1[u,v] \mid (u,v)\in B_0\cup B_1\}$. In particular $|\SPP^-|=|\SPP^+|$. More importantly,
\[ |\SPP^-| \;\le\; |\SPP_0(e)| + |\SPP_1(e)|\;.\]
%
%
To bound $\Exp[|\SPP^-|]=\Exp[|\SPP^+|]$ it is thus enough to bound $\Exp[|\SPP_0(e)|]=\Exp[|\SPP_1(e)|]$.

\begin{lemma}\label{L-length} The expected number of edges on a shortest path between two random vertices is $(1+o(1))\ln n$.
\end{lemma}

\begin{proof} When edge weights are exponential, the expected number of edges on a shortest path between two random vertices is exactly equal to the average depth of a vertex in a random recursive tree of size~$n$. (See, e.g., Janson \cite{Janson99}.) It is known that this average depth is $(1+o(1))\ln n$ (Moon \cite{Moon74}). The same asymptotic result holds also under the uniform distribution. (See Section~2 of Janson \cite{Janson99}.)
\end{proof}

\begin{lemma} The expected number of shortest paths that pass through a random edge $e$ is $(1+o(1))\ln n$.
\end{lemma}

\begin{proof} For every $u,v\in V$, let $\pi[u,v]$ be the shortest path from~$u$ to~$v$, and let $|\pi[u,v]|$ be the number of edges on it. For every edge $e$ of the complete graph, let $\SPP(e)$ be the set of shortest paths that pass through~$e$. 
By symmetry we have
$$\Exp[|\SPP(e)|] = E\biggl[\frac{1}{n(n-1)}\sum_{e'} |\SPP(e')|\biggr] = E\biggl[\frac{1}{n(n-1)}\sum_{u\ne v} |\pi[u,v]|\biggr]=\Exp[|\pi[u,v]|].$$
By Lemma~\ref{L-length}, we get that
$\Exp[|\SPP(e)|] = (1+o(1))\ln n$.
\end{proof}


\begin{theorem} Following a random edge update, we have $\Exp[|\SPP^-|]=\Exp[|\SPP^+|]\le (2+o(1))\ln n$.
\end{theorem}

Let $\Delta$ be the maximal degree of the essential graph $G^*=(V,E^*)$ defined in the previous section. Lemma~\ref{L-Delta} says that with high probability $\Delta=O(\log n)$.


\begin{theorem}\label{T-ELSP} Following a random edge update we have $\Exp[|\LSP^-|]=\Exp[|\LSP^+|]=O(\log^2 n)$.
\end{theorem}

\begin{proof} Clearly $\pi\in \LSP^-$ if and only if $l[\pi]\in \SPP^-$ or $r[\pi]\in \SPP^-$. Each shortest path has at most $2\Delta$ LSP extensions. Thus $|\LSP^-|\le 2\Delta \cdot|\SPP^-|$.
By Lemma~\ref{L-Delta}, we have $\Pr[\Delta>24\ln n]=O(n^{-3})$.  As $|\LSP|$ is always at most~$n^3$, we get
$\Exp[|\LSP^-|]\le 48\ln n\cdot \Exp[|\SPP^-|]+ n^{-3}\cdot n^3 = O(\log^2 n)$.
\end{proof}

We believe that the $O(\log^2 n)$ bound in Theorem~\ref{T-ELSP} can be improved, possibly to $O(\log n)$, and leave it as an open problem.

\begin{theorem} The expected running time of a random edge update, when a Fibonacci heap is used to implement the global heap, and simple linked lists are used to implement the local heaps, is $O(\log^2 n)$.
\end{theorem}




\section{Concluding remarks}\label{S-concl}

We presented an algorithm that solves the APSP problem on complete directed graphs with random edges weights in $O(n^2)$ time with high probability. The expected running time of the algorithm is also $O(n^2)$. This solves an open problem of Frieze and McDiarmid \cite{FrMc97}.

We also presented a dynamic algorithm that performs random edge updates in $O(\log^2 n)$ expected time. It is an interesting open problem whether this can be improved to $O(\log n)$.

Our results also hold in the directed $G(n,p)$ model in which each edge is selected with probability~$p$, where $p\gg (\ln n)/n$. Selected edges are again assigned independent, uniformly distributed, weights.
%
Similarly, it is easy to see that our results apply when edge weights are \emph{integers} chosen uniformly at random from, say, $\{1,2,\ldots,n\}$, where $n$ is the number of vertices.

\section*{Acknowledgment} The last author would like to thank Camil Demetrescu, Giuseppe Italiano and Mikkel Thorup for many illuminating discussions.



\makeatletter
\def\runninghead{\hrulefill\quad APPENDIX\quad\hrulefill}
\def\ps@headings{
\def\@oddhead{\footnotesize\rm\hfill\runninghead\hfill}}
\def\@evenhead{\@oddhead}
\def\@oddfoot{\rm\hfill\thepage\hfill}\def\@evenfoot{\@oddfoot}
\makeatother

\appendix

\section{The static algorithm -- complete description}\label{S-static}

\newcommand{\APSP}{
\parbox{3.2in}{
\begin{function}[H]
\SetVline \dontprintsemicolon

\BlankLine
$\init(G)$ \;
$Q\gets \heap()$ \;

\BlankLine
\ForEach{$(u,v)\in E$}
{
    $dist[u,v] \gets \weight(u,v)$ \;
    $p[u,v] \gets v$ \;
    $q[u,v] \gets u$ \;
    $\heapinsert(Q,(u,v),dist[u,v])$ \;
}

\BlankLine
\While{$Q\ne\emptyset$} {
    $(u,v) \gets \extractmin(Q)$ \;
    $\INSERT(L[p[u,v],v],u)$ \;
    $\INSERT(R[u,q[u,v]],v)$ \;
    \ForEach{$w\in L[u,q[u,v]]$}
    {
        $\EXAMINE(w,u,v)$
    }
    \ForEach{$w\in R[p[u,v],v]$}
    {
        $\EXAMINE(u,v,w)$
    }
}

\caption{\apsp(\mbox{$G=(V,E,\weight)$})}
\end{function}
} }

\newcommand{\INIT}{
\parbox{3.2in}{
\begin{function}[H]
\SetVline \dontprintsemicolon

\BlankLine
\ForEach{$u,v\in V$}
{
    $dist[u,v]\gets \infty$ \;
    $p[u,v]\gets \NULL$ \;
    $q[u,v]\gets \NULL$ \;
    $L[u,v]\gets \emptyset$ \; 
    $R[u,v]\gets \emptyset$ \;
}
\ForEach{$u\in V$}
{
    $dist[u,u]\gets 0$ \;
}

\caption{\init(\mbox{$G=(V,E,\weight)$})}
\end{function}

\vspace*{10pt}
 \begin{function}[H]
 \SetVline \dontprintsemicolon
\BlankLine
 \If{$dist[u,v]+dist[v,w]<dist[u,w]$}
 {
     $dist[u,w] \gets dist[u,v]+dist[v,w]$ \;
     \eIf{$p[u,w]=\NULL$} 
         {$\heapinsert(Q,(u,w),dist[u,w])$}
         {$\decreasekey(Q,(u,w),dist[u,w])$}
     $p[u,w] \gets p[u,v]$ \;
     $q[u,w] \gets q[v,w]$ \;
 }

 \caption{\EXAMINE(\mbox{$u,v,w$})}
 \end{function}
 }
 }

In this section we give a full description, and a correctness proof, of the static version of the Demetrescu and Italiano \cite{DeIt04a,DeIt06} used in this paper.
%
%
Pseudo-code of the algorithm, called \apsp, is given in
Figure~\ref{F-apsp}. The input to the algorithm is a weighted directed graph $G=(V,E,c)$, where $c:E\to (0,\infty)$ assigns positive weights (or costs) to the edges of the graph. The algorithm in Figure~\ref{F-apsp} works correctly only under the assumption that all shortest paths are \emph{unique}. Under essentially all probabilistic models considered in this paper, this assumption holds with probability~$1$. Algorithm \apsp\ is also interesting, however, in non-probabilistic settings.
For a simple way of dispensing with the uniqueness assumption, without increasing the running time of the algorithm by more than a constant factor, see Demetrescu and Italiano \cite{DeIt04a}.

We next prove Theorem~\ref{T-APSP-2} of Section~\ref{S-DI}, which we repeat for the convenience of the reader.

{
\renewcommand{\thetheorem}{\ref{T-APSP-2}}
\begin{theorem}
If all edge weights are positive and all shortest paths are unique, then algorithm \apsp\ correctly finds all the shortest paths in the graph.
Algorithm \apsp\ runs in $O(n^2\cdot(T_{ins}(n^2)+T_{ext}(n^2)) + |\LSP|\cdot T_{dec}(n^2))$ time, where $|\LSP|$ is the number of LSPs in the graph, and uses only $O(n^2)$ space.
\end{theorem}
\addtocounter{theorem}{-1}
}


\begin{figure}[b]
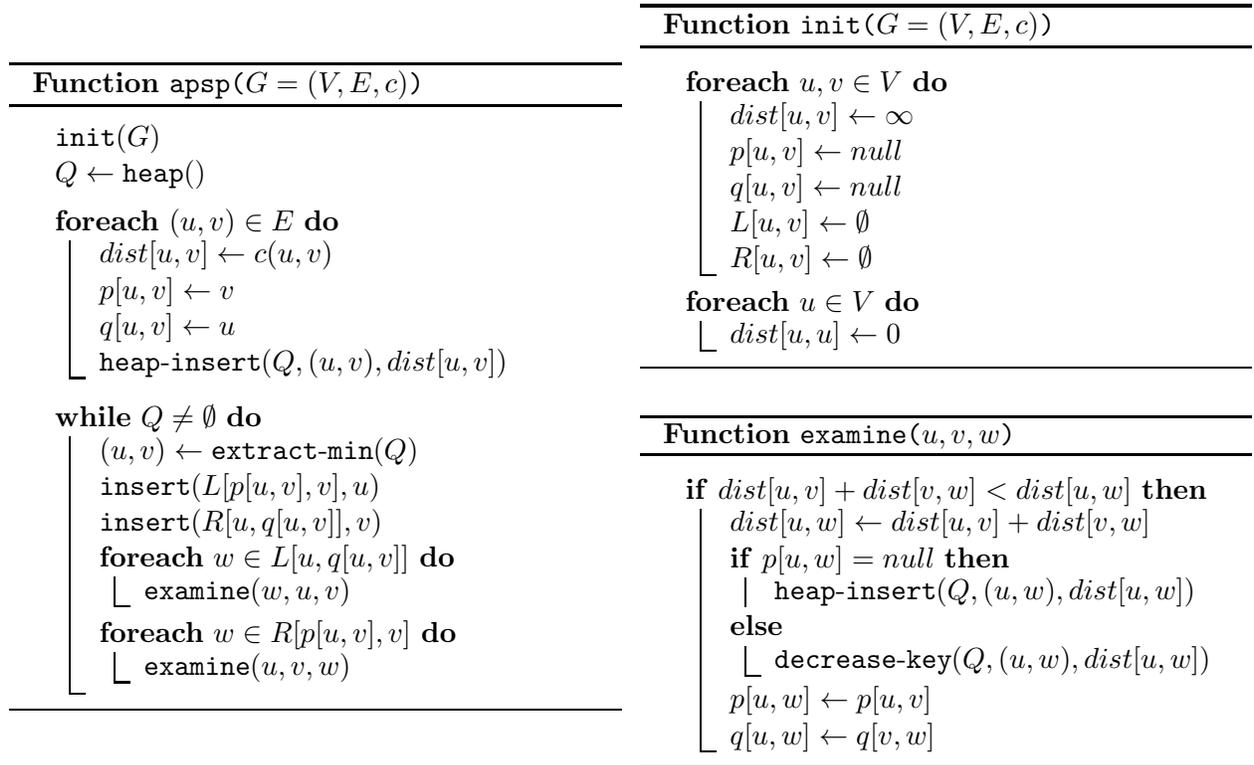

\begin{center}
\vspace*{10pt}
\APSP \INIT
\end{center}
\vspace*{-10pt} \caption{A static version of the APSP algorithm of Demetrescu and Italiano \cite{DeIt04a,DeIt06}.} \label{F-apsp}
\end{figure}

\begin{proof} It is easy to check that each stage during the operation of the algorithm, $dist[u,v]$ corresponds to some path from~$u$ to~$v$ in the graph and that this path, or an even shorter path, can be traced using the~$p$ and~$q$ fields. Thus, the distances returned by the algorithm can never be too small.

It is also easy to check that the keys of the pairs $(u,v)$ extracted from~$Q$ form a non-decreasing sequence and that a pair $(u,v)$ removed from~$Q$ is never inserted to~$Q$ again. Thus, the algorithm always terminates.

Assume, for the sake of contradiction, that the algorithm fails to find a shortest path $u\SP v$, for some $u,v\in V$. Let $u\SP v$ be a shortest shortest-path not found by the algorithm. (In other words, if $u'\SP v'$ is shorter than $u\SP v$, then $u'\SP v'$ is found by the algorithm.)

If $u\SP v$ is simply the edge $u\to v$, then we immediately get a contradiction, as the algorithm starts by setting $dist[u,v]$ to $c(u,v)$ (and $p[u,v]$ to $v$, and $q[u,v]$ to~$u$), for every $(u,v)\in E$. Thus, the algorithm does find the shortest path $u\SP v=u\to v$, a contradiction.

Assume, therefore, that $u\SP v=u\to u'\SP v'\to v$ is composed of at least two edges. (If it is composed of exactly two edges, then $u'=v'$.) Clearly $u\to u'\SP v'$ and $u'\SP v'\to v$ are also shortest paths and their length is strictly smaller than the length of $u\SP v$, as $c(u,u'),c(v',v)>0$. Thus, by the our assumptions, $u\to u'\SP v'$ and $u'\SP v'\to v$ are discovered by the algorithm. When the second of these is discovered, the algorithm examines the path $u\SP v=u\to u'\SP v'\to v$ and sets $dist[u,v]$ to its length. Also $(u,v)$ is added to~$Q$ if it is not already there. As there is no shorter path from $u$ to~$v$ in the graph, the values of $dist[u,v]$, $p[u,v]$ and $q[u,v]$ would never be changed again, contradicting the assumption that the algorithm does not find the shortest path from~$u$ to~$v$.


We next analyze the running time of algorithm. Each pair $(u,v)$ is inserted and extracted from the priority queue~$Q$ at most once. The total cost of these operations is $O(n^2(T_{ins}(n^2)+T_{ext}(n^2)))$. All paths considered by the algorithm are LSPs. The algorithm examines each LSP exactly once. For each LSP it performs a constant number of operations followed perhaps by a \decreasekey\ operation. The total cost of all these operations is $O(|\LSP|\,T_{dec}(n^2))$. The complexity of all other operations is negligible.

Finally, to see that the algorithm uses only $O(n^2)$ space, note that
the removal of a pair $(u,v)$ from the heap~$Q$ causes the insertion of only two elements to lists $L[u',v']$ and $R[u',v']$. As each pair $(u,v)$ is extracted at most once, the total size of all these lists is $O(n^2)$.
\end{proof}

\section{The dynamic algorithm -- complete description}\label{S-dynamic}

\newcommand{\PATHV}{
\parbox{1.8in}{
\begin{function}[H]
\SetVline \dontprintsemicolon

 $\pi\gets \newpath()$ \;
 $l[\pi]\gets \NULL$ \;
 $r[\pi]\gets \NULL$ \;
 $start[\pi]\gets v$ \;
 $end[\pi] \gets v$ \;
 $first[\pi]\gets \NULL$ \;
 $last[\pi] \gets \NULL$ \;
 $cost[\pi] \gets 0$ \;
 $sp[\pi]\gets \TRUE$ \;
 $L[\pi],R[\pi]\gets \emptyset$ \;
 $SL[\pi],SR[\pi]\gets \emptyset$ \;
 \Return{$\pi$}

\caption{\PATH(\mbox{$v$})}
\end{function}
} }

\newcommand{\PATHE}{
\parbox{2.3in}{
\begin{function}[H]
\SetVline \dontprintsemicolon

 $\pi\gets \newpath()$ \;
 $l[\pi]\gets p[u]$ \;
 $r[\pi]\gets p[v]$ \;
 $start[\pi]\gets u$ \;
 $end[\pi] \gets v$ \;
 $first[\pi]\gets e$ \;
 $last[\pi] \gets e$ \;
 $cost[\pi] \gets c[e]$ \;
 $sp[\pi]\gets \FALSE$ \;
 $L[\pi],R[\pi]\gets \emptyset$ \;
 $SL[\pi],SR[\pi]\gets \emptyset$ \;
 $\INSERT(L[p[v]],\pi)$ \;
 $\INSERT(R[p[u]],\pi)$ \;
 \Return{$\pi$}

\caption{\PATH(\mbox{$e=(u,v)$})}
\end{function}
} }

\newcommand{\PATHP}{
\parbox{2.6in}{
\begin{function}[H]
\SetVline \dontprintsemicolon

 \lIf{$r[\pi_1]\ne l[\pi_2]$}{error} \;
 $\pi\gets \newpath()$ \;
 $l[\pi]\gets \pi_1$ \;
 $r[\pi]\gets \pi_2$ \;
 $start[\pi]\gets start[\pi_1]$ \;
 $end[\pi] \gets end[\pi_2]$ \;
 $first[\pi]\gets first[\pi_1]$ \;
 $last[\pi] \gets last[\pi_2]$ \;
 $cost[\pi] \gets c[first[\pi]]+cost[\pi_2]$ \;
 $sp[\pi]\gets \FALSE$ \;
 $L[\pi],R[\pi]\gets \emptyset$ \;
 $SL[\pi],SR[\pi]\gets \emptyset$ \;
 $\INSERT(L[\pi_2],\pi)$ \;
 $\INSERT(R[\pi_1],\pi)$ \;
 \Return{$\pi$}

\caption{\PATH(\mbox{$\pi_1,\pi_2$})}
\end{function}
} }

\newcommand{\DINIT}{
\parbox{3in}{
\begin{function}[H]
\SetVline \dontprintsemicolon

\BlankLine
$Q\gets \heap()$ \;

\BlankLine
\ForEach{$u,v\in V$}
{
    $\pi[u,v] \gets \NULL$ \;
    $dist[u,v] \gets \infty$ \;
    $P[u,v]\gets \heap()$ \;
}

\BlankLine
\ForEach{$u\in V$}
{
    $p[u] \gets \PATH(u)$ \;
    $\pi[u,u] \gets p[u]$ \;
    $dist[u,u] \gets 0$ \;
}

\BlankLine
\ForEach{$e\in E$}
{
    $p[e] \gets \PATH(e)$ \;
    $\examine(p[e])$
}

\BlankLine
$\buildpaths()$
\caption{\dapspinit(\mbox{$G=(V,E,c)$})}
\end{function}
}
}


\newcommand{\BUILDPATHS}{
\parbox{3in}{
\begin{function}[H]
\SetVline \dontprintsemicolon

\While{$Q\ne \emptyset$}
{
    $(u,v) \gets \extractmin(Q)$ \;
    $\newshortestpath(\pi[u,v])$
}

\caption{\buildpaths($\,\!$)}
\end{function}
}}

\newcommand{\INITBUILDPATHS}{
\parbox{3in}{
\begin{function}[H]
\SetVline \dontprintsemicolon

\ForEach{$u\ne v\in V$}
{
    $dist[u,v] \gets \infty$ \;
    $\pi[u,v] \gets \NULL$ \;
    \If{$P[u,v]\ne\emptyset$}
    {
       $\pi\gets \findmin(P[u,v])$ \;
       $dist[u,v] \gets cost[\pi]$ \;
       $\pi[u,v] \gets \pi$
       $\heapinsert(Q,(u,v),dist[u,v])$
    }

}
\caption{\initbuildpaths($\,\!$)}
\end{function}
}
}


\newcommand{\NEWSHORTPATH}{
\parbox{3in}{
\begin{function}[H]
\SetVline \dontprintsemicolon

\BlankLine
$sp[\pi]\gets \TRUE$ \;

\BlankLine
$\INSERT(SL[r[\pi]],\pi)$ \;
$\INSERT(SR[l[\pi]],\pi)$ \;

\BlankLine
\ForEach{$\pi'\in SL[l[\pi]]$} {
    $\pi'' \gets \PATH(\pi',\pi)$ \;
    $\examine(\pi'')$
}

\BlankLine
\ForEach{$\pi'\in SR[r[\pi]]$} {
    $\pi'' \gets \PATH(\pi,\pi')$ \;
    $\examine(\pi'')$
}

\caption{\newshortestpath(\mbox{$\pi$})}
\end{function}
}}

\newcommand{\DEXAMINE}{
\parbox{3in}{
\begin{function}[H]
\SetVline \dontprintsemicolon
    \BlankLine
    $u \gets start[\pi]$ ; $v\gets end[\pi]$ \;
    $\heapinsert(P[u,v],\pi,cost[\pi])$ \;


    \BlankLine
    \If{$cost[\pi]<dist[u,v]$}
    {
        \BlankLine
        \If{$\pi[u,v]\ne \NULL$}
        {
            $sp[\pi[u,v]] \gets \FALSE$ \;
            $\removeextensions(\pi[u,v],\FALSE)$
        }
        \BlankLine
        $\pi[u,v] \gets \pi$ \;
        $dist[u,v] \gets cost[\pi]$ \;
        $\heapinsert(Q,(u,v),cost[\pi])$
    }

\caption{\examine(\mbox{$\pi$})}
\end{function}

}

}


\newcommand{\INSERTEDGES}{
\parbox{2.8in}{
\begin{function}[H]
\SetVline \dontprintsemicolon

\ForEach{$e=(u,v)\in E_{ins}$}
{
    $\INSERT(E[u],e)$ \;
    $\INSERT(E[v],e)$ \;
    $p[e] \gets \PATH(e)$ \;
    $\heapinsert(P[u,v],p[e],cost[e])$
}
\caption{\insertedges(\mbox{$E_{ins}$})}
\end{function}

\vspace*{10pt}

\begin{function}[H]
\SetVline \dontprintsemicolon
\ForEach{$e=(u,v) \in E_{ins}$}
{
    $\DELETE(E[u],e)$ \;
    $\DELETE(E[v],e)$ \;
    $\removepath(p[e],\TRUE)$
}

\caption{\deleteedges(\mbox{$E_{del}$})}
\end{function}
} }


\newcommand{\REMOVEPATH}{
\parbox{3in}{
\begin{function}[H]
\SetVline \dontprintsemicolon

\BlankLine
$u\gets start[\pi]$ ; $v\gets end[\pi]$ \;
$\heapdelete(P[u,v],\pi)$ \;


\BlankLine
    $\DELETE(R[l[\pi]],\pi)$ \;
    $\DELETE(L[r[\pi]],\pi)$ \;

    \BlankLine
    \If{$sp[\pi]=\TRUE$}
    {
    \If{$rep=\TRUE$}
    {
        $\INSERT(A,(u,v))$
    }
    $\DELETE(SR[l[\pi]],\pi)$ \;
    $\DELETE(SL[r[\pi]],\pi)$ \;
    }

\BlankLine
$\removeextensions(\pi,rep)$ \;

\caption{\removepath(\mbox{$\pi,rep$})}
\end{function}
}
}


\newcommand{\REMOVEEXTENSIONS}{
\parbox{3in}{
\begin{function}[H]
\SetVline \dontprintsemicolon
\BlankLine
    \ForEach{$\pi'\in L[\pi]\cup R[\pi]$}
    {
        $\removepath(\pi',rep)$
    }
\caption{\removeextensions(\mbox{$\pi,rep$})}
\end{function}
}
}


\newcommand{\REPLACEPATH}{
\parbox{3in}{
\begin{function}[H]
\SetVline \dontprintsemicolon

\BlankLine
\eIf{$P[u,v]\ne \emptyset$}
{
    $\pi \gets \findmin(P[u,v])$ \;
    $\pi[u,v] \gets \pi$ \;
    $dist[u,v] \gets cost[\pi]$ \;
    $\heapinsert(Q,(u,v),cost[\pi])$ \;
}
{
    $\pi[u,v] \gets \NULL$ \;
    $dist[u,v] \gets \infty$
}

\caption{\replacepath(\mbox{$u,v$})}
\end{function}
}
}

\newcommand{\UPDATE}{
\parbox{3.0in}{
\begin{function}[H]
\SetVline \dontprintsemicolon

\BlankLine
$A\gets \emptyset$ \;
\ForEach{$e\in E'$}
{
    $\removepath(p[e],\TRUE)$ \;
}

\BlankLine
$Q \gets \heap()$ \;
\ForEach{$(u,v)\in A$}
{
    $\replacepath(u,v)$
}

\BlankLine
\ForEach{$e\in E'$}
{
    $c[e]\gets c'[e]$ \;
    $p[e] \gets \PATH(e)$ \;
    $\examine(p[e])$
}

\BlankLine
$\buildpaths()$

\caption{update(\mbox{$E',c'$})}
\end{function}
}
}


\begin{figure}[t]
\begin{center}
\PATHV \PATHE 
\PATHP
\end{center}
\vspace*{-10pt}
\caption{Generating new paths and inserting it into the path system.}
\label{F-PATH}
\end{figure}

As explained, one of the main differences between the static and dynamic algorithms
is that the dynamic algorithm explicitly maintains all LSPs in a \emph{path system}, and does not just examine them. Paths are created by the three constructors $\PATH(v)$, $\PATH(e)$ and $\PATH(\pi_1,\pi_2)$ given in Figure~\ref{F-PATH}.
$\PATH(v)$ generates a path of length $0$ containing the vertex~$v$. $\PATH(e)$ generates a path composed of the edge~$e$. $\PATH(\pi_1,\pi_2)$ takes two paths~$\pi_1$ and~$\pi_2$ such that $r[\pi_1]=l[\pi_2]$ and constructs a path~$\pi$ such that $l[\pi]=\pi_1$ and $r[\pi]=\pi_2$. The new path~$\pi$ is composed of the first edge of~$\pi_1$ followed by~$\pi_2$, or equivalently, by $\pi_1$ followed by the last edge of~$\pi_2$.

Every path $\pi$ has the following fields:

\smallskip
$\quad l[\pi]$ - A pointer to the path obtained by removing the last edge of~$\pi$.

$\quad r[\pi]$ - A pointer to the path obtained by removing the first edge of~$\pi$.

$\quad start[\pi]$ - The first vertex on~$\pi$.

$\quad end[\pi]$ - The last vertex on~$\pi$.

$\quad first[\pi]$ - The first edge on~$\pi$.

$\quad last[\pi]$ - The last edge on $\pi$.

$\quad cost[\pi]$ - The total cost (weighted length) of $\pi$.

$\quad sp[\pi]$ - \TRUE\ if and only if $\pi$ is known to be a shortest path.

$\quad L[\pi]$ - List of left LSP extensions of~$\pi$.

$\quad R[\pi]$ - List of right LSP extensions of~$\pi$.

$\quad SL[\pi]$ - List of left shortest path extensions of~$\pi$.

$\quad SR[\pi]$ - List of right shortest path extensions of~$\pi$.

\smallskip
The lists $SL[\pi]$ and $SR[\pi]$ are similar to the lists $L[u,v]$ and $R[u,v]$ used by the static algorithm. This time, however, they contain actual paths and not vertices. The lists $L[\pi]$ and $R[\pi]$ contain all LSPs, already constructed, obtained by extending $\pi$ by one edge at its beginning or end, respectively.

The initialization function of the dynamic version, called \dapspinit, is given in Figure~\ref{F-update}. It is similar to the static \apsp\ algorithm. It too uses a global heap~$Q$ that stores pairs of vertices for which shortest paths are sought.
For every $v\in V$, we let $p[v]$ be the empty path consisting of~$v$.
For every edge $e\in E$, we let~$p[e]$ denote the path consisting of~$e$.
For every two vertices $u,v\in V$, the dynamic algorithm maintains the following information:

\smallskip
$\quad \pi[u,v]$ - The shortest path from~$u$ to~$v$ found so far.

$\quad cost[u,v]$ - The cost of the shortest path from~$u$ to~$v$ found so far.

$\quad P[u,v]$ - a heap containing all the LSPs from~$u$ to~$v$ found so far.

\smallskip
We refer to $P[u,v]$ as the \emph{local heap} corresponding to the pair $(u,v)$. We refer to~$Q$ as the \emph{global heap}.

The initialization function \dapspinit\ starts
with some obvious initializations. (For every $u,v\in V$, it sets~$\pi[u,v]$ to $\NULL$, sets $dist[u,v]$ to $\infty$, sets $P[u,v]$ to an empty heap, etc.) For every $e\in E$ it then creates the path $p[e]$, by calling $\PATH(e)$, and then \emph{examines} it by calling $\examine(p[e])$, given in Figure~\ref{F-examine}.

The function $\examine(\pi)$ receives a newly created LSP connecting two vertices~$u=start[\pi]$ and~$v=end[\pi]$. It starts by inserting it into the heap $P[u,v]$ with key $cost[\pi]$. It then checks whether $\pi$ is the first available LSP from~$u$ to~$v$, or whether it is shorter than all existing LSPs between $u$ and~$v$.
If $\pi$ is shorter than~$\pi[u,v]$, the shortest available path from~$u$ to~$v$, then
$\pi[u,v]$ is clearly not a shortest path. The algorithm thus sets $sp[\pi[u,v]]$ to $\FALSE$.
It then removes all extensions of~$\pi[u,v]$ from the system, if there are any. This is done by a call to $\removeextensions(\pi[u,v],\FALSE)$ which we discuss later.
Finally, if $\pi$ is currently the shortest available path from~$u$ to~$v$, \examine\ updates $\pi[u,v]$ and $dist[u,v]$ accordingly. It also inserts $(u,v)$ into the global heap, if it is not already there, or decreases its key to $cost[\pi]$. (We assume that $\heapinsert$ does exactly that, i.e., inserts an item into a heap with a given key, or decreases its key, if the item is already in the heap.)

\dapspinit\ then calls \buildpaths\ which is also given in Figure~\ref{F-update}. \buildpaths\ repeatedly removes a pair $(u,v)$ with the smallest key from the global heap~$Q$. The corresponding path $\pi[u,v]$ is then a shortest path. The call $\newshortestpath(\pi)$ is then made.

The function $\newshortestpath(\pi)$ receives a newly discovered shortest path. It sets to $sp[\pi]$ to \TRUE. It inserts $\pi$ to the lists $SL[r[\pi]]$ and $SR[r[\pi]]$, as $\pi$ is now a shortest path left extension of $r[\pi]$ and a shortest path right extension of $l[\pi]$. (Note that $\pi$ is already contained in $L[r[\pi]]$ and $R[l[\pi]]$ at this stage.) Most importantly, $\newshortestpath(\pi)$ now constructs LSPs extensions of~$\pi$ and examines each one of them.
(These operations may add new pairs into the global heap~$Q$.)

Using essentially the same arguments used to prove Theorem~\ref{T-APSP-2}, we get that \dapspinit\
correctly finds all shortest and locally shortest paths in the graph.

Updates are performed by calling \update, also given in Figure~\ref{F-update}.
$\update(E',c')$ assigns the edges of~$E'$ new edges weights and recomputes all shortest paths. 
$\update(E',c')$ starts by removing all paths that pass through edges of~$E'$. This done by calling $\removepath(p[e],\TRUE)$, for every \mbox{$e\in E'$}. (Function $\removepath$ is discussed below.) These removals create a list~$A$ of pairs $(u,v)$ that lost their shortest path. For every $(u,v)\in A$, a call is made to $\replacepath(u,v)$. Paths corresponding to all edges of~$E'$ are recreated, with their new costs, and these edge paths are examined. All updated shortest paths are then obtained by a call to \buildpaths.

$\replacepath(u,v)$, given in Figure~\ref{F-removepaths}, receives a pair of vertices $(u,v)$ such that the shortest path from~$u$ to~$v$ has just been destroyed. It finds the shortest path~$\pi$ in $P[u,v]$, if there is one, and performs the necessary updates. (Note that $\pi$ is not necessarily the shortest path from~$u$ to~$v$. It is just the shortest path currently available.)

Finally, paths and their extensions are removed from the path system by the functions \removepath\ and \removeextensions\ also given in Figure~\ref{F-removepaths}. To remove a path~$\pi$ from the path system, $\removepath(\pi,rep)$ deletes $\pi$ from $P[u,v]$, where $u=start[\pi]$ and $v=end[\pi]$ are the endpoints of~$\pi$. It also deletes~$\pi$ from $R[l[\pi]]$ and $L[r[\pi]]$. If~$\pi$ is marked as a shortest path, i.e., $sp[\pi]=\TRUE$, then~$\pi$ is also removed from $SR[l[\pi]]$ and $SL[r[\pi]]$. Finally, if $sp[\pi]=\TRUE$ and $rep=\TRUE$, then $(u,v)$ is inserted into a list~$A$ of pairs who lost their shortest paths. $\removeextensions(\pi,rep)$ removes all the extensions of~$\pi$ from the path system, by calling $\removepath(\pi')$, for every $\pi'\in L[\pi]\cap R[\pi]$.


Theorem~\ref{T-dynamic} now follows by examining the operation of the algorithm. When a shortest path is destroyed it is removed from its local heap. In some cases, the shortest path in the local heap is found and a pair $(u,v)$ is inserted into the global heap. The total cost of these operations is at most $T_{del}(\Lambda)+T_{min}(\Lambda)+T_{ins}(n^2)$, where $\Lambda$ is an upper bound on the size of the local heaps. Each new shortest path is extracted from the global heap at a total cost of $T_{ext}(n^2)$. Each LSP destroyed is removed from its local heap at a cost of $T_{del}(\Lambda)$.
Finally, each LSP created is inserted into the appropriate local heap and possibly causes a decrease-key operation on the global heap, a total cost of $T_{ins}(\Lambda)+T_{dec}(n^2)$.

\begin{figure}[t]
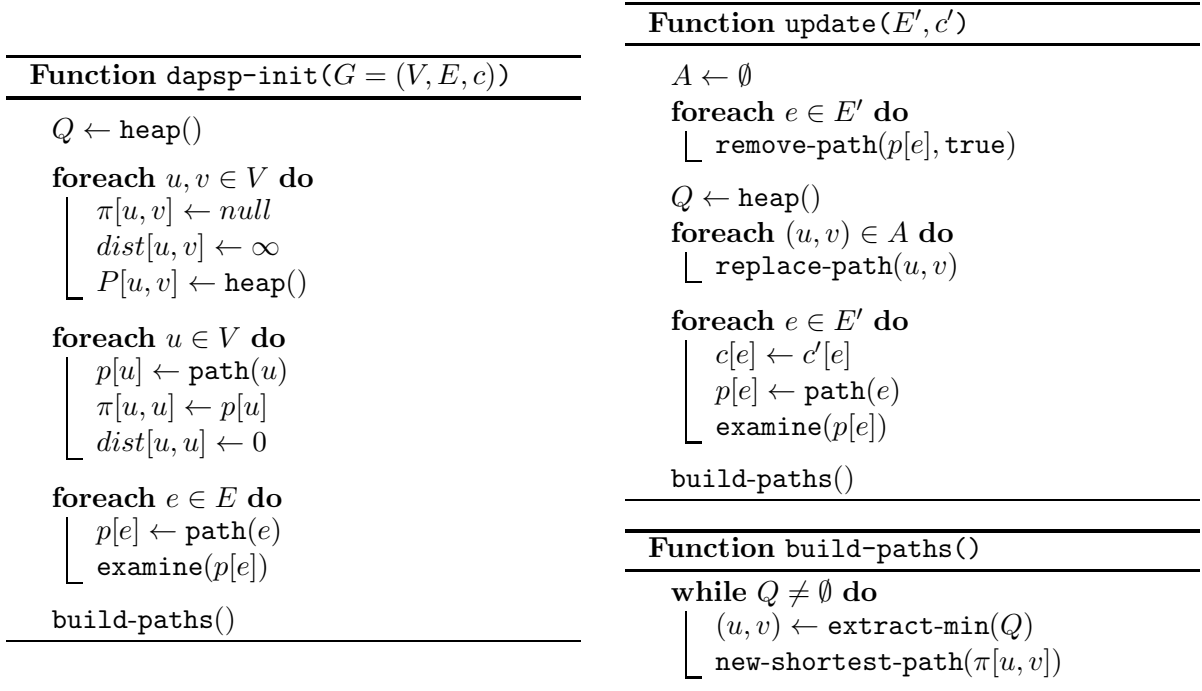

\begin{center}
\DINIT \hspace*{10pt}
\parbox{3in}{\UPDATE

\vspace*{3pt}
\BUILDPATHS}
\end{center}
\vspace*{-10pt}
\caption{Initiating and updating the dynamic all-pairs shortest paths data structure.}
\label{F-update}
\end{figure}

\begin{figure}[t]
\begin{center}
\NEWSHORTPATH \hspace*{0.8cm} \DEXAMINE
\end{center}
\vspace*{-10pt}
\caption{The functions \newshortestpath\ and \examine.}\label{F-examine}
\end{figure}

\begin{figure}[t]
\begin{center}
\REMOVEPATH \hspace*{10pt}
\parbox{3in}{
\REMOVEEXTENSIONS

\vspace*{2pt}
\REPLACEPATH}
\end{center}
\vspace*{-10pt}
\caption{Removing a path and its extensions from the path system.}
\label{F-removepaths}
\end{figure}

\end{document}